\documentclass[a4paper]{amsart}

\usepackage[utf8]{inputenc}
\usepackage[T1]{fontenc}
\usepackage{lmodern, enumerate}
\usepackage{amssymb,amsxtra}
\usepackage[all]{xy}
\usepackage{nicefrac,mathtools}
\usepackage{microtype}
\usepackage{amscd}
\usepackage{xcolor}
\usepackage[pdftitle={Maximality of dual coactions and applications},
 pdfauthor={Alcides Buss, Siegfried Echterhoff},
 pdfsubject={Mathematics}
]{hyperref}
\usepackage[lite]{amsrefs}
\newcommand*{\MRref}[2]{ \href{http://www.ams.org/mathscinet-getitem?mr=#1}{MR #1}}
\newcommand*{\arxiv}[1]{\href{http://www.arxiv.org/abs/#1}{arXiv: #1}}

\numberwithin{equation}{section}
\theoremstyle{plain}
\newtheorem{theorem}[equation]{Theorem}
\newtheorem{lemma}[equation]{Lemma}
\newtheorem{proposition}[equation]{Proposition}
\newtheorem{corollary}[equation]{Corollary}
\theoremstyle{definition}
\newtheorem{definition}[equation]{Definition}
\theoremstyle{remark}
\newtheorem{remark}[equation]{Remark}
\newtheorem{example}[equation]{Example}

\DeclareMathOperator{\Aut}{Aut}

\DeclareMathOperator{\spn}{span}

\DeclareMathOperator{\supp}{\mathrm{supp}}

\newcommand*{\nb}{\nobreakdash}
\newcommand*{\Star}{\(^*\)\nobreakdash-}
\newcommand*{\dd}{\,d}

\newcommand*{\C}{\mathbb C}

\newcommand*{\Real}{\mathbb R}
\newcommand*{\NN}{\mathbb N}

\newcommand*{\Lb}{\mathcal L}
\newcommand*{\K}{\mathcal K}
\renewcommand*{\H}{\mathcal H}
\newcommand*{\M}{\mathcal{M}} 

\newcommand*{\cont}{C}
\newcommand*{\contz}{\cont_0}
\newcommand*{\contc}{\cont_c}

\newcommand*{\id}{\textup{id}}
\newcommand*{\Ad}{\textup{Ad}}

\renewcommand{\i}{\operatorname{i}}

\newcommand*{\E}{\mathcal E}
\newcommand*{\A}{\mathcal A}
\newcommand*{\B}{\mathcal B}
\newcommand*{\EE}{\mathbb E}
\newcommand*{\BB}{\mathbb B}

\newcommand*{\defeq}{\mathrel{\vcentcolon=}}
\newcommand*{\congto}{\xrightarrow\sim}

\newcommand*{\sbe}{\subseteq} 
\newcommand*{\F}{\mathcal F}
\newcommand*{\cstar}{\texorpdfstring{$C^*$\nobreakdash-\hspace{0pt}}{*-}}
\newcommand*{\into}{\hookrightarrow}
\newcommand*{\onto}{\twoheadrightarrow}
\newcommand*{\red}{{r}}
\renewcommand*{\max}{\mathrm{max}}

\newcommand{\eps}{\varepsilon}

\newcommand*{\Free}{\mathbb{F}}

\newcommand{\lk}{\langle}
\newcommand{\rk}{\rangle}

\newcommand{\dualG}{\widehat{G}}
\newcommand{\dual}[1]{\widehat{#1}}

\begin{document}
\title[Maximality of dual coactions and applications]{Maximality of dual coactions on sectional  \\ \cstar{}algebras
of Fell bundles and applications}

\author{Alcides Buss}
\email{alcides@mtm.ufsc.br}
\address{Departamento de Matem\'atica\\
 Universidade Federal de Santa Catarina\\
 88.040-900 Florian\'opolis-SC\\
 Brazil}

\author{Siegfried Echterhoff}
\email{echters@uni-muenster.de}
\address{Mathematisches Institut\\
 Westf\"alische Wilhelms-Universit\"at M\"un\-ster\\
 Einsteinstr.\ 62\\
 48149 M\"unster\\
 Germany}

\begin{abstract}
In this paper we give a simple proof of the maximality of dual coactions on full cross-sectional \cstar{}algebras of Fell bundles over locally compact groups. This result was only known for discrete groups or for saturated (separable) Fell bundles over locally compact groups.
Our proof, which is derived as an application of the theory of universal generalised fixed-point algebras for weakly proper actions, is different from these  previously known cases and works for general Fell bundles over locally compact groups.
As  applications we extend certain exotic crossed-product functors in the sense of Baum, Guentner and Willett to the
category of Fell bundles and the category of partial actions and we obtain results about the $K$-theory  of (exotic) cross-sectional algebras of Fell-bundles over $K$-amenable groups.
As a bonus, we give a characterisation of maximal coactions of discrete groups in terms of maximal tensor products.
\end{abstract}

\subjclass[2010]{46L55, 46L08}

\thanks{Supported by Deutsche Forschungsgemeinschaft (SFB 878, Groups, Geometry \& Actions) and by CNPq/CAPES (Brazil).}

\maketitle

\section{Introduction}
\label{sec:introduction}
The theory of Fell bundles $\mathcal B$ over a locally compact group $G$ (also called \cstar{}algebraic bundles in \cite{Doran-Fell:Representations_2})
and their cross-sectional algebras give far reaching generalisations of the theory of
crossed products by strongly continuous actions $\alpha:G\to\Aut(A)$ of $G$ on \cstar{}algebras $A$.
Important examples of Fell bundles come from (twisted) partial actions (see \cite{Exel:TwistedPartialActions}) of $G$
on \cstar{}algebras $A$ and in this case the crossed products for such actions are by definition
given as the cross-sectional \cstar{}algebras of the associated Fell bundles.

Recall from {\cite{Doran-Fell:Representations,Doran-Fell:Representations_2}} that  a Fell bundle $\mathcal B$ over $G$ consists of
a topological space $\mathcal B$ together with a continuous open surjection
$p:\mathcal B\onto G$ such that the fibres $B_s:=p^{-1}(\{s\})$ are Banach spaces for all $s\in G$ and such that
all operations like
multiplication with scalars,  fibre-wise addition,  and norm are continuous on $\B$.
Moreover, $\mathcal B$
 comes equipped with  an associative continuous
multiplication function
$$\cdot: \mathcal B\times \mathcal B\to\mathcal B; (a,b)\mapsto a\cdot b$$
which is bilinear when restricted to $B_s\times B_t$ for all $s,t\in G$
and such that $B_s\cdot B_t\subseteq B_{st}$.
In addition, $\mathcal B$ is equipped with a continuous involution
$*:\mathcal B\to\mathcal B$, $b\mapsto b^*$ which sends
$B_t$ to $B_{t^{-1}}$ for all $t\in G$ and which is compatible
with multiplication and addition on $\mathcal B$ in a sense
extending the usual properties for involutions on \cstar{}algebras.
In particular, the \cstar{}condition $\|b^*b\|=\|b\|^2$ and the positivity condition $b^*b\geq 0$ in $B_e$
are required to hold for all $b\in \B$. Note that the unit fibre $B_e$ in a Fell bundle $\B$ is always a \cstar{}algebra. A Fell bundle is called {\em saturated} if $\spn(B_t^*B_t)=B_e$ for all $t\in G$.

Given a Fell bundle $\mathcal B$, let $C_c(\B)$ denote the
space of continuous sections with compact support. It carries
multiplication and involution given by the formulas
\begin{equation}\label{eq-conv}
f*g(s)=\int_G f(t)g(t^{-1}s)\,dt\quad\text{and}\quad f^*(s)=\Delta_G(s^{-1}) f(s^{-1})^*.
\end{equation}
In general there might exist many possible \cstar{}completions of $C_c(\B)$.
The largest ($L^1$\nb-bounded) \cstar{}norm on $C_c(\B)$ is the {\em universal (or maximal) cross-sectional algebra} $C^*(\B)$ whose
representations are in one-to-one correspondence to the continuous \Star{}representations of the bundle $\B$. On the other extreme there is the {\em reduced} cross-sectional algebra $C_r^*(\B)$ which is defined as the image of $C^*(\B)$
under the regular representation $\Lambda: \B\to \Lb(L^2(\B))$.

If $\B$ is a Fell-bundlle and if $u_G:G\to U\M(C^*(G))$ denotes the universal representation of $G$, then the
integrated form of the  representation
$$\delta_\B=\iota_\B\otimes u_G: \B\to \M(C^*(\B)\otimes C^*(G)); b_t\mapsto b_t\otimes u_G(t)$$
defines a {\em dual coaction} of $G$ (or rather of the Hopf-\cstar{}algebra $C^*(G)$)
on $C^*(\B)$  (see \cite{ExelNg:ApproximationProperty}).
It is the main purpose  of this paper to show that this coaction always satisfies Katayama duality for the maximal bidual crossed
product in the sense that a certain canonical surjective homomorphism
$$\Phi_\B: C^*(\B)\rtimes_{\delta_\B}\widehat{G}\rtimes_{\widehat{\delta_\B}}G\onto C^*(\B)\otimes \K(L^2(G))$$
will actually be an isomorphism. Coactions with this property have been called {\em maximal} in \cite{Echterhoff-Kaliszewski-Quigg:Maximal_Coactions}, where
it has first been shown that every coaction $(B,\delta)$ admits a {\em maximalisation} $(B_m,\delta_m)$.
If $G$ is discrete and $(B, \delta)$ is any coaction of $G$ on some \cstar{}algebra $B$, then it follows from results of Ng and Quigg in \cite{Ng:Discrete-coactions,Quigg:Discrete-coactions} (see also \cite{Echterhoff-Quigg:InducedCoactions}) that $B$ is isomorphic to a \cstar{}completion $C_\mu^*(\B)$ of $C_c(\B)$ for some Fell bundle $\B$ with respect to some norm
$\|\cdot\|_\mu$ lying between the universal norm $\|\cdot\|_u$ and the reduced norm $\|\cdot\|_r$ such that
the coaction $\delta$ is the natural dual coaction of this algebra. It has then been shown in \cite{Echterhoff-Quigg:InducedCoactions} (see also \cite[\S4]{Echterhoff-Kaliszewski-Quigg:Maximal_Coactions}) that $\delta$ is maximal if and only if $B=C^*(\B)$, the universal cross-sectional algebra. This gives the desired result in the discrete case.

For second countable locally compact groups the result that $(C^*(\B),\delta_\B)$ is a maximal coaction has been obtained in the special case of \emph{separable saturated} Fell bundles by Kaliszewski, Muhly, Quigg and Williams in \cite{Kaliszewski-Muhly-Quigg-Williams:Coactions_Fell}.
Separability is not a strong assumption, but note that Fell bundles arising from (twisted) partial actions are saturated if and only if the action
is actually a global (twisted) action, so that there are many important examples of Fell bundles which are not saturated.
Moreover, the proof given in \cite{Kaliszewski-Muhly-Quigg-Williams:Coactions_Fell} relies on some heavy machinery about Fell bundles on groupoids while our proof depends on  the notion of generalised fixed-point algebras for weakly proper actions as introduced
recently in \cite{Buss-Echterhoff:Exotic_GFPA}. Both proofs depend on non-trivial results, but we believe that our proof is much shorter and technically easier than the proof given for the special case of saturated bundles in \cite{Kaliszewski-Muhly-Quigg-Williams:Coactions_Fell}. We should point out, however, that for the special case of Fell bundles associated to partial actions, the maximality result can also be deduced from the paper \cite{Abadie-Perez:Amenability} by Fernando Abadie and Laura Mart\'i P\'erez. Indeed, as we will see in Section~\ref{Sec:Partial-Actions}, the maximality of the dual coaction on the \cstar{}algebra $C^*(\B)$ of a Fell bundle associated to a partial action $(A,\alpha)$ is essentially equivalent to the fact (proven in \cite{Abadie-Perez:Amenability}) that the full crossed product of the Morita enveloping action of $(A,\alpha)$ is Morita equivalent to the original full crossed product by the partial action. Hence our main result will provide an alternative proof of one of the main results in \cite{Abadie-Perez:Amenability}.

The paper is organised as follows. After a short preliminary section (\S2) on cross-sectional algebras, coactions, and generalised fixed-point algebras for weakly proper actions we shall give the proof of our main result (Theorem~\ref{theo:maximal-dual-coaction}) in \S3. We will then have a number of interesting applications, starting from extensions of crossed-product functors from ordinary actions to Fell bundle categories, $K$\nb-amenability for cross-sectional \cstar{}algebras and some applications to partial actions. Our results imply, in particular, that one can extend exotic crossed-product functors from ordinary actions to partial actions. In other words, given a crossed-product functor $(A,\alpha)\mapsto A\rtimes_{\alpha,\mu}G$ defined only for (global) $G$-actions $(A,\alpha)$, we can extend this functor and define $A\rtimes_{\alpha,\mu}G$ for every given partial $G$-action $(A,\alpha)$. More generally, we can extend the functor to the category of Fell bundles over $G$, and define exotic versions $C^*_\mu(\B)$ of cross-sectional \cstar{}algebras of Fell bundles $\B$ over $G$. These include partial actions or, more generally, twisted partial actions.
We give an alternative proof and recover some of the main results on enveloping actions and amenability for partial actions from \cite{Abadie-Perez:Amenability}. More generally, we prove that all exotic cross-sectional \cstar{}algebras $C^*_\mu(\B)$ have same $K$-theory if the underlying group is $K$-amenable.

In the final section, we give a simple characterisation of maximal coactions of discrete groups (which is not available for general locally compact groups), proving that a coaction $\delta\colon B\to B\otimes C^*(G)$ of a discrete group $G$ is maximal exactly when it admits a lift $\delta_\max\colon B\to B\otimes_\max C^*(G)$, where $\otimes_\max$ denotes the maximal tensor product. (And, as usual, all the unlabeled tensor products $\otimes$ between \cstar{}algebras mean the minimal tensor product.)

This work started during a visit of the second author to Florian\'opolis to participate in the FADYS (Functional Analysis and Dynamical Systems) Workshop. It was during the talk of the second author in this Workshop that Ruy Exel asked the question whether the theory of exotic crossed products could be extended to cover partial actions as well. One of the main points of this paper is to give a positive answer to this question. We would like to thank Ruy Exel for asking this interesting question and make this paper emerge.
The second author takes this opportunity to thank the members of the Department of Mathematics of UFSC for organising the workshop and for their warm
hospitality during his stay in Florian\'opolis.

\section{Preliminaries}\label{sec-prel}
\subsection{Fell bundles and their cross-sectional algebras.}
Suppose that $p:\B\to G$ is a Fell bundle over the locally compact group $G$ as
defined in the introduction. Main references on Fell bundles and their cross-sectional algebras are the books by Doran and Fell\cite{Doran-Fell:Representations,Doran-Fell:Representations_2} and we refer to these books for  more details of the definition
(see also the book by Exel \cite{Exel:Book}). Let $C_c(\B)$ be the set of all continuous sections
of $\B$ with compact supports. Then, equipped with  convolution and involution as in (\ref{eq-conv}),
 $C_c(\B)$ becomes a \Star{}algebra. Let $L^1(\B)$ denote the completion of $C_c(\B)$ with respect to
 $\|f\|_1=\int_G\|f(t)\|\,dt$. Then the universal cross-sectional algebra $C^*(\B)$ is defined as the enveloping
 \cstar{}algebra of the Banach-* algebra $L^1(\B)$, i.e., it is the completion of $L^1(\B)$ with respect to
 $$\|f\|_u:=\sup_{\pi}\|\pi(f)\|$$
 where $\pi$ runs through all \Star{}representations of $L^1(\B)$ on Hilbert space. It has been shown by Fell
 \cite[\S 15--16]{Fell:Induced} that $C^*(\B)$ is universal for (continuous) \Star{}representations of $\B$, and since we shall
 need it later, let us explain (a modern version) of this result in some detail:
By a  {\em \Star{}representation} $\pi:\B\to \M(D)$ of $\B$ into the multiplier algebra
of  some \cstar{}algebra $D$ we understand a strictly continuous map $b\mapsto \pi(b)$
which is linear on each fibre $B_t$ and preserves multiplication and involution on $\B$.
Such representation is called {\em nondegenerate} if its restriction $\pi_e: B_e\to \M(D)$
on the unit fibre $B_e$ of $\B$ is nondegenerate in the usual sense that $\spn(\pi_e(B_e)D)$ is dense in $D$
(then Cohen's factorisation theorem  implies that $\pi_e(B_e)D=D$).
There is a canonical nondegenerate representation $\iota_\B:\B\to\M(C^*(\B))$ which is determined by the
formulas
$$(\iota_\B(b_t) f)(s)=b_tf(t^{-1}s)\quad\text{and}\quad (f\iota_\B(b_t))(s)=\Delta_G(t^{-1})f(st^{-1})b_t$$
 for $b_t\in B_t$, $f\in C_c(\B)$ and $t,s\in G$ (e.g., see \cite[p. 138]{Fell:Induced}).
We have the following well-known result:

 \begin{proposition}\label{prop-rep}
There are one-to-one correspondences between
\begin{enumerate}
\item nondegenerate representations $\pi: \B\to \M(D)$,
\item nondegenerate \Star{}representations $\tilde\pi: C_c(\B)\to \M(D)$ which are continuous
with respect to the inductive limit topology on $C_c(\B)$ and the norm-topology on $\M(D)$,
\item nondegenerate  \Star{}representations $\tilde\pi: L^1(\B)\to \M(D)$, and
\item nondegenerate \Star{}representations  $\tilde\pi: C^*(\B)\to \M(D)$.
\end{enumerate}
\end{proposition}

The correspondences in the proposition are given as follows:
If $\pi:\B\to \M(D)$ is as in (1), then the corresponding representation
$\tilde\pi:C_c(\B)\to \M(D)$ is given
by integration: $\tilde\pi(f)c:=\int_G \pi(f(t))c\,dt$ for all $f\in C_c(\B)$ and $c\in D$.
We call it the {\em integrated form} of $\pi$.
It is straightforward to check that it is continuous in the inductive limit topology
(given by uniform convergence with controlled compact supports) and with respect
to $\|\cdot \|_1$. This gives (1) $\rightarrow$ (2),(3).
By construction of the enveloping \cstar{}algebra,
every \Star{}representation of $L^1(\B)$ uniquely extends to $C^*(\B)$ which gives (3) $\rightarrow$ (4).
Conversely, if $\tilde\pi: C^*(\B)\to \M(D)$ is as in (4), then $\pi=\tilde\pi\circ \iota_B$ is the corresponding
representation as in (1). The only missing link is the connection (2) $\rightarrow$ (1). But this follows
from \cite[Theorem 16.1]{Fell:Induced} by representing $D$ faithfully on a Hilbert space.

In what follows, we shall make no notational difference between a representation $\pi$
of $\B$ and the corresponding representations of $C_c(\B), L^1(\B)$, and $C^*(\B)$.

Let $L^2(\mathcal B)$ denote the Hilbert module over $B_e$ given as the completion of $C_c(\mathcal B)$ with respect to the $B_e$-valued inner product
$$\lk \xi, \eta\rk_{B_e}:= (\xi^**\eta)(e) =\int_G\xi(t)^*\eta(t)\dd{t}.$$
Then the action of $C_c(\mathcal B)$ on itself given by convolution extends to
the
{\em regular representation} $\Lambda_{\mathcal B}:C^*(\B)\to\mathcal L(L^2(\mathcal B))$
by adjointable operators on the $B_e$-Hilbert module $L^2(\mathcal B)$ and the image
$C_r^*(\B):=\Lambda_\B(C^*(\B))\subseteq \mathcal L(L^2(\B))$ is called the
{\em reduced} cross-sectional \cstar{}algebra of $\B$.

\subsection{Coactions and their crossed products}
A coaction of a locally compact group on a \cstar{}algebra $B$ is a  nondegenerate \Star{}homomorphism
$\delta:B\to \M(B\otimes C^*(G))$ which satisfies the identity
\begin{equation}\label{eq-coact} (\id_B\otimes \delta_G)\circ \delta= (\delta\otimes \id_G)\circ \delta,\end{equation}
where $\delta_G: C^*(G)\to \M(C^*(G)\otimes C^*(G))$ is the comultiplication on $C^*(G)$ which is given
by the integrated form of the representation $s\mapsto u(s)\otimes u(s)\in U\M(C^*(G)\otimes C^*(G))$, where
$u:G\to U\M(C^*(G))$ denotes the universal representation of $G$.
Note that in (\ref{eq-coact}) (and in many other places) we make no notational difference between a nondegenerate \Star{}homomorphism and its unique extension to the multiplier algebra.
We shall assume that our coactions $\delta$ always satisfy the following (strong) nondegeneracy condition
$$\delta(B)(1\otimes C^*(G))=B\otimes C^*(G).$$
If $\delta: B\to \M(B\otimes C^*(G))$ is a coaction of $G$ we let
 $\delta^r: = (1\otimes\lambda)\circ \delta: B\to \M(B\otimes C_r^*(G))$ denote the {\em reduction} of $\delta$,
where $\lambda: G\to U(L^2(G))$ is the regular representation of $G$,  and we let $M:C_0(G)\to \BB(L^2(G))$
be the representation by multiplication operators.
We then may represent the crossed product $B\rtimes_\delta\widehat{G}$ faithfully in
$\M(B\otimes \K(L^2(G)))$ via the regular representation
$\Lambda^{\widehat{G}}_{B}:=\delta^r\rtimes (1\otimes M)$.
Hence, via this representation, we may define
$$B\rtimes_{\delta}\widehat{G}:=\overline{\spn}\left\{\delta^r(B)\big(1\otimes M(C_0(G))\big)\right\}
\subseteq \M(B\otimes \K(L^2(G))).$$
Since locally compact groups are always ``co-amenable'' this ``reduced crossed product'' coincides
with the ``universal crossed product'' which is universal for covariant representations of the co-system
$(B,\delta)$. Note that in the notation of crossed products by coactions we use the symbol $\widehat{G}$ to indicate
that this construction is {\em dual} to the construction of crossed products by actions of $G$. We refer to
\cite[Appendix A]{Echterhoff-Kaliszewski-Quigg-Raeburn:Categorical} for an extensive survey on crossed products by actions and coactions of groups on \cstar{}algebras.

The \emph{dual action} $\widehat{\delta}:G\to \Aut(B\rtimes_{\delta}\widehat{G})$ is determined by the equation
$$\widehat{\delta}_s\big(\delta^r(b)(1\otimes M(\varphi))\big)=\delta^r(b)(1\otimes M(\sigma_s(\varphi))),\quad \forall b\in B, \varphi\in C_0(G),$$
 where $\sigma:G\to \Aut(C_0(G))$ denotes the {\em right translation action}.
One checks  that $(\Lambda_{B}^{\widehat{G}}, 1\otimes \rho)$ is a covariant representation
of the dual system $(B\rtimes_{\delta}\widehat{G}, G, \widehat{\delta})$ on $\M(B\otimes \K(L^2(G)))$
and it follows from \cite[Corollary 2.6]{Nilsen} that the integrated form of this representation gives a surjective \Star{}homomorphism
$$\Phi_B\defeq\big((\id\otimes\lambda)\circ \delta_\A\rtimes(1\otimes M)\big)\rtimes (1\otimes \rho)\colon B\rtimes_{\delta}\widehat{G}\rtimes_{\widehat{\delta}}G\onto B\otimes \K(L^2(G)).$$
The map $\Phi_B$ might be called the {\em Katayama-duality map}.
Now, following \cite{Echterhoff-Kaliszewski-Quigg:Maximal_Coactions} a coaction $(B,\delta)$ is called {\em maximal} if the homomorphism $\Phi_B$ is an isomorphism.

On the other extreme, a coaction $(B,\delta)$ of $G$
 is called {\em normal} if the surjection $\Phi_B$ factors through an isomorphism
$$B\rtimes_{\delta}\widehat{G}\rtimes_{\widehat{\delta},r}G\cong B\otimes \K(L^2(G)).$$
It has been shown by Quigg in \cite{Quigg:Discrete-coactions} that every coaction $(B,\delta)$ has a {\em normalisation} $(B_n,\delta_n)$, which can be constructed by passing from $B$ to the
quotient $B_n:=B/\ker\delta^r$.  In particular, it follows that $(B,\delta)$ is normal if and only if its reduction $\delta^r$ is injective. On the other hand it has been shown in \cite{Echterhoff-Kaliszewski-Quigg:Maximal_Coactions} that every coaction also
 has a  {\em maximalisation} $(B_m,\delta_m)$ such that there exist $\widehat{G}$-equivariant surjections
$B_m\onto B\onto B_n$
which induce isomorphisms between the respective coaction-crossed products.

For later use we need the construction of the maximalisation and normalisation of  $(B,\delta)$
as given in \cite{Buss-Echterhoff:Exotic_GFPA}, using the notion of generalised fixed-point algebras for weakly proper actions.
In what follows let us write $(j_B, j_{C_0(G)})$  for the covariant representation $(\delta^r, 1\otimes M)$ when viewed as
a  representation into $\M(B\rtimes_\delta \widehat{G})$. It is then clear that
$j_{C_0(G)}:C_0(G)\to \M(B\rtimes_\delta \widehat{G})$ is a nondegenerate $\sigma-\widehat{\delta}$-equivariant
\Star{}homomorphism which gives $(B\rtimes_{\delta}\widehat{G}, G, \widehat{\delta})$ the structure of a
weakly proper $G\rtimes G$-algebra in the sense of \cite{Buss-Echterhoff:Exotic_GFPA}. For simpler notation let us write
 $A:=B\rtimes_{\delta}\widehat{G}$ and we put $\varphi\cdot m:=j_B(\varphi)m$ and $m\cdot\varphi:=mj_B(\varphi)$
 for all $m\in \M(A), \varphi\in C_0(G)$. Moreover, let $A_c:=C_c(G)\cdot A\cdot C_c(G)$, which is a dense \Star{}subalgebra
 of $A$, and let
 \begin{equation}
 A_c^G:=\{m\in \M(A)^G: m\cdot \varphi, \varphi\cdot m\in A_c\; \forall \varphi\in C_c(G)\},
 \end{equation}
where $\M(A)^G$ denotes the set of fixed-points in $\M(A)$ for the extended action $\widehat{\delta}$.
 We call $A_c^G$ the {\em generalised fixed-point algebra with compact supports}.
 Following ideas of Rieffel (\cite{Rieffel:Proper, Rieffel:Integrable_proper}), it has then been shown in \cite[Proposition 2.2]{Buss-Echterhoff:Exotic_GFPA} that
 $\F_c(A):=C_c(G)\cdot A$ can be made into a pre-Hilbert $C_c(G,A)$ module by defining a $C_c(G,A)$-valued inner product on $\F_c(A)$ and a
 right action of $C_c(G,A)$ on $\F_c(A)$ by
 $$\lk \xi, \eta\rk_{C_c(G,A)}=\big[s\mapsto \Delta_G(s)^{-\frac12}\xi^*\widehat{\delta}_s(\eta)\big]
 \quad\text{and}\quad\xi\cdot \varphi=\int_G\Delta(t)^{-\frac12}\alpha_t(\xi\varphi(t^{-1}))\,dt$$
for $\xi,\eta\in \F_c(A)$ and $\varphi\in C_c(G,A)$.
Let $A\rtimes_{\widehat{\delta}, \mu}G$ be any \cstar{}completion of $C_c(G,A)$ with respect to a \cstar{}norm
$\|\cdot\|_\mu$ on $C_c(G,A)$ such that $\|\cdot\|_u\geq\|\cdot\|_\mu\geq \|\cdot\|_r$, where $\|\cdot\|_u$ and $\|\cdot\|_r$ denote
the universal (i.e.,  maximal) and the reduced norm on $C_c(G,A)$, respectively.
Then the above defined inner product takes values in $A\rtimes_{\widehat\delta,\mu}G$ and
the completion $\F_\mu(A)$ of $\F_c(A)$ with respect to this  inner product
 becomes a full $A\rtimes_{\widehat\delta, \mu}G$-Hilbert module (the module is full since the translation action of $G$ on itself
 is free and proper).
Now, if we  define a left action of $A_c^G$ on $\F_c(A)$ by taking products inside $\M(A)$ (where both spaces are
located), this action extends to a faithful \Star{}homomorphism $\Psi_\mu: A_c^G\to\K(\F_\mu(A))$ with dense image.
Hence $A_\mu^G:=\K(\F_\mu(A))$ can be viewed as the completion of $A_c^G$ with respect to the
operator norm for the left action of $A_c^G$ on $\F_\mu(A)$. In particular, $\F_\mu(A)$ becomes a
$A_\mu^G-A\rtimes_{\widehat{\delta},\mu}G$-equivalence bimodule.

Moreover, if the dual coaction on $A\rtimes_{\widehat\delta}G$ factors through a dual coaction on $A\rtimes_{\widehat{\delta},\mu}G$
 (a property which depends on the norm $\|\cdot\|_\mu$), it is shown in
\cite[Theorem 4.6]{Buss-Echterhoff:Exotic_GFPA} that there are canonical coactions $\delta_{A_\mu^G}$ and $\delta_{\F_\mu(A)}$ of $G$
on $A_\mu^G$ and  $\F_\mu(A)$, respectively,  such that
$(\F_\mu(A), \delta_{\F_\mu(A)})$ becomes a $\widehat{G}$-equivariant Morita equivalence between
$(A_\mu^G, \delta_{A_\mu^G})$ and $(A\rtimes_{\widehat{\delta},\mu}G, \widehat{\widehat\delta\,})$.
It is shown in \cite[Lemma 4.8]{Buss-Echterhoff:Exotic_GFPA} that there exists a unique crossed-product norm $\|\cdot\|_\mu$
on $C_c(G,A)$ such that $(A_\mu^G, \delta_{A_\mu^G})$ is isomorphic to the original coaction $(B,\delta)$.
Moreover, if  $\|\cdot\|_\mu=\|\cdot\|_u$ is the universal norm on $C_c(G,A)$, then the
corresponding system  $(B_m,\delta_m):=(A_u^G, \delta_{A_u^G})$ is a maximalisation for $(B,\delta)$ and
if $\|\cdot\|_\mu=\|\cdot\|_r$ is the reduced norm, then
 $(B_n,\delta_n):=(A_r^G, \delta_{A_r^G})$  is a normalisation of
$(B,\beta)$.  Identifying $(B,\beta)$ with $(A_\mu^G, \delta_{A_\mu^G})$ as above, the
identity map on $A_c^G$ induces the $\widehat{G}$-equivariant surjections $B_m\onto B\onto B_n$
which induce isomorphisms of crossed products
$$B_m\rtimes_{\delta_m}\widehat{G}\cong B\rtimes_{\delta}\widehat{G}\cong B_n\rtimes_{\delta_n}\widehat{G}.$$

\section{The main result}
Assume that $p:\B\to G$ is a Fell bundle over the locally compact group $G$. Then there is a canonical coaction
$$\delta_\B: C^*(\B)\to \M(C^*(\B)\otimes C^*(G)),$$
called the \emph{dual coaction} of $G$ on $C^*(\B)$, given as the integrated form of the \Star{}re\-pre\-sen\-ta\-tion $\delta_\B:\B\to \M(C^*(\B)\otimes C^*(G))$ which sends $B_t\ni b_t\mapsto \iota_\B(b_t)\otimes u_t$, where $\iota_\B:\B\to \M(C^*(\B))$ is the
universal representation of $\B$ and  $u:G\to U\M(C^*(G))$ is the universal representation of $G$.

\begin{theorem}\label{theo:maximal-dual-coaction}
Let $\B$ be a Fell bundle over the locally compact group $G$.
Then the dual coaction $\delta_\B\colon C^*(\B)\to \M(C^*(\B)\otimes C^*(G))$ is maximal.
\end{theorem}

Let $(B,\delta)\defeq (C^*(\B),\delta_\B)$ and let $(B_m,\delta_m)$ be the maximalisation of $(B,\delta)$ as
constructed from $(B,\delta)$ in the previous section. We will show that there exists a $\delta-\delta_m$-equivariant
surjection $\Psi: B\onto B_m$ which induces an isomorphism of crossed products $B\rtimes_\delta \widehat{G}\cong B_m\rtimes_{\delta_m}\widehat{G}$. The result will then follow from the following easy lemma, which should be well known to the experts:

\begin{lemma}\label{lem-max} Let $(B,\delta)$ and $(B_m,\delta_m)$ be coactions of $G$ with $(B_m,\delta_m)$ maximal.
Suppose  that $\Psi:B\onto B_m$ is a
$\delta-\delta_m$-equivariant surjection which induces an isomorphism of crossed products.
Then $\Psi$ is an isomorphism and $(B,\delta)$ is maximal as well.
\end{lemma}
\begin{proof} Since $\Phi:B\onto B_m$ is $\delta-\delta_m$-equivariant, we obtain a commutative diagram
$$
\begin{CD} B\rtimes_{\delta}\widehat{G}\rtimes_{\widehat\delta}G @> \Phi_B >> B\otimes \K(L^2(G))\\
@V \Psi\rtimes\widehat{G}\rtimes G VV                       @VV \Psi\otimes\id_{\K(L^2(G))} V\\
B_m\rtimes_{\delta_m}\widehat{G}\rtimes_{\widehat\delta_m}G @> \Phi_{B_m} >> B_m\otimes \K(L^2(G)).
\end{CD}
$$
By our assumptions, the left vertical  and the lower horizontal arrows are isomorphisms. It then follows that
the upper horizontal arrow has to be injective. Since it is always surjective, it must be an isomorphism.
Hence $(B,\delta)$ is maximal. Moreover, it follows that the right vertical arrow is an isomorphism which
then implies that $\Psi:B\to B_m$ must be an isomorphism as well.
\end{proof}

\begin{proof}[Proof of Theorem \ref{theo:maximal-dual-coaction}]
Let $(B,\delta)\defeq (C^*(\B),\delta_\B)$ and let
$$A:=B\rtimes_\delta \dualG=\overline{\spn}\left\{\delta^r(B)\big(1\otimes M(C_0(G))\big)\right\}$$
As explained in the previous section, we view $A$ as a weakly proper $G\rtimes G$-algebra.
Then, as explained above, the maximalisation of $(B,\delta)$ is given by the coaction $(B_m,\delta_m)=(A_u^G, \delta_{A_u^G})$
where $A_u^G$ denotes the universal generalised fixed-point algebra of $A$.
We will show that the restriction of $\delta^r $ to $C_c(\B)$ defines a \Star{}homomorphism
$\Psi: C_c(\B)\to A_c^G\subseteq \M(A)$ which extends to the desired $\delta-\delta_m$-equivariant
 surjective \Star{}homomorphism $\Psi: C^*(\B)\onto A_u^G$.
 First of all, it follows directly from the definition of the dual action $\widehat{\delta}$ that $\delta^r(B)$ lies in the
 fixed-point algebra $\M(A)^G$.  To see that it sends $C_c(\B)$ into the generalised fixed-point algebra $A_c^G$
 with compact supports it suffices to show that all elements of the form
 $\delta^r(b)(1\otimes M(f)), (1\otimes M(f))\delta^r(b)$ lie in $A_c=C_c(G)\cdot A\cdot C_c(G)$ for
 all $b\in C_c(\B)$ and $f\in C_c(G)$. For this
 we
first note  that $\delta^r=(1\otimes \lambda)\circ \delta_\B$
 is the integrated form of the representation $\delta^r:\B\to \M(C^*(B)\otimes \K(L^2(G))); b_t\mapsto \iota_\B(b_t)\otimes \lambda(t)$.
 Suppose now that $b\in C_c(\B)$ is a continuous section with compact support $K=\supp(b)$.
 Then, if $f\in C_c(G)$ is fixed,  we may choose a function $g\in C_c(G)$ such that $g\equiv 1$ on $K\cdot \supp(f)\cup\supp(f)$.
 Then for $a=\delta^r(b)(1\otimes M(f))$ we clearly have $a\cdot g=\delta^r(b)(1\otimes M(fg))=\delta^r(b)(1\otimes M(f))=a$.
On the other side, using  $\lambda_tM(g)\lambda_{t^{-1}}=M(\tau_t(g))$, where $\tau:G\to \Aut(C_0(G))$
denotes the left translation action of $G$ on itself, we compute
\begin{equation}\label{eq-Ac}
 \begin{split}
 g\cdot a&=(1\otimes M(g))\delta^r(b)(1\otimes M(f))\\
 &=\int_K (1\otimes M(g))(\iota_\B(b_t)\otimes \lambda_t)(1\otimes M(f))\,dt\\
 &=\int_K (\iota_\B(b_t)\otimes \lambda_t)(1\otimes M(\tau_{t^{-1}}(g)f))\,dt\\
 &= \int_K  (\iota_\B(b_t)\otimes \lambda_t)(1\otimes M(f))\,dt=a
 \end{split}
 \end{equation}
 since  for $t\in K$ and $s\in \supp(f)$ we have $\tau_{t^{-1}}(g)(s)=g(ts)=1$ since $ts\in K\cdot\supp(f)$.
 This proves that $\delta^r(C_c(\B))(1\otimes M(f))$ lies in $A_c$ and a similar argument also gives that
 $(1\otimes M(f))\delta^r(b)\in A_c$.

 Now we need to show that $\delta^r: C_c(\B)\to A_c^G$ extends to an equivariant surjective \Star{}homomorphism
 $\Psi:C^*(\B)\to A_u^G=\K(\F_u(A))$.
 For this we need to recall from \cite[Definition 2.6]{Buss-Echterhoff:Exotic_GFPA}
 the notion of convergence in the inductive limit topology on the spaces $A_c=C_c(G)\cdot A\cdot C_c(G)$,
 $\F_c(A)=C_c(G)\cdot A$ and $A_c^G$, respectively.  First of all, a sequence
 $(\xi_n)_{n\in \NN}$ in $\F_c(A)$ (resp. $A_c$) converges to $\xi\in \F_c(A)$ (resp. $\xi\in A_c$)
 in the inductive limit topology, if $\xi_n\to \xi$ in the norm topology of $A$ and  there exists a  function
 $g\in C_c(G)$ such that $\xi=g\cdot\xi$,  $\xi_n=g\cdot\xi_n$
 (resp. $\xi=g\cdot\xi\cdot g$,  $\xi_n=g\cdot\xi_n\cdot g$) for all $n\in \NN$.
For $A_c^G$, a sequence $(m_n)_{n\in \NN}$ in $A_c^G$ converges to $m\in A_c^G$ in the inductive limit topology
if for all $f\in C_c(G)$ we have $f\cdot m_n\to f\cdot m$ and $m_n\cdot f\to m\cdot f$ in the inductive limit
topology of $A_c$ (the fact that $G/G$ is a one-point set implies that this definition coincides with the one given in
\cite[Definition 2.6]{Buss-Echterhoff:Exotic_GFPA}). Now it is shown in \cite[Lemma 2.7]{Buss-Echterhoff:Exotic_GFPA} that all pairings in the
$A_c^G - C_c(G,A)$ pre-imprimitivity bimodule $\F_c(A)$ are jointly continuous with respect to the
inductive limit topologies, where on $C_c(G,A)$ we use the usual notion of inductive limit convergence.
Since inductive limit convergence in $C_c(G,A)$ is stronger than norm convergence with respect to any
given \cstar{}norm $\|\cdot\|_\mu$ on $C_c(G,A)$, it follows from this that the inductive limit topology on $A_c^G$ is
stronger than any norm topology induced on $A_c^G$ via the left action on the $A\rtimes_{\widehat{\delta},\mu}G$-Hilbert module
$\F_\mu(A)$. In particular,  inductive limit convergence in $A_c^G$ implies norm convergence
in $A_u^G$.

Assume now that $(b_n)_{n\in \NN}$ is a sequence in $C_c(\B)$ which converges to some $b\in C_c(\B)$ in
the inductive limit topology of $C_c(\B)$ (which means that $b_n\to b$ uniformly on $G$ and that
all $b_n$ have  supports in a fixed compact subset $K$ of $G$). Then the computation in (\ref{eq-Ac}) can easily
be modified to show that $\delta^r(b_n)\to \delta^r(b)$ in the inductive limit topology of $A_c^G$.
Thus $\delta^r(b_n)\to \delta^r(b)$ in the universal completion $A_u^G$. Thus we obtain a \Star{}representation
$\Psi: C_c(\B)\to A_u^G$ which is continuous for the inductive limit topology on $C_c(\B)$. But then Proposition \ref{prop-rep}
implies that $\Psi$ extends to a \Star{}homomorphism $\Psi: C^*(\B)\to A_u^G$.

To see that the image is dense, we first show that
$$\E_c:=\spn\big(\delta^r(C_c(\B))(1\otimes M(C_c(G)))\big)\subseteq A_c$$
 is
 inductive limit dense in $A_c$. Since $\E_c$ is norm dense in $A$, it is clear
 that
 $$\spn\big((1\otimes M(C_c(G)))\delta^r(C_c(\B))(1\otimes M(C_c(G)))\big)$$
  is inductive limit dense
 in $A_c$. Hence it suffices to show that every element of the form
 $(1\otimes M(g))\delta^r(b)(1\otimes M(f))$ with $f,g\in C_c(G)$ and $b\in C_c(\B)$ can be inductive limit approximated by
 elements in $\E_c$. By (\ref{eq-Ac}) we know that
 $$(1\otimes M(g))\delta^r(b)(1\otimes M(f))=\int_K (\iota_\B(b_t)\otimes \lambda_t)(1\otimes M(\tau_{t^{-1}}(g)f))\,dt,$$
 where $K=\supp(b)$. Now, for each $\eps>0$ and $t\in K$ we find a neighbourhood $V_t$ of $t$ in $G$
 such that $\|\tau_{s^{-1}}(g)f -\tau_{t^{-1}}(g)f\|_\infty<\eps$ for all $s\in V_t$. Let
 $t_1, \ldots , t_n$ be given such that $K\subseteq \cup_{l=1}^n V_{t_l}$, let $\varphi_1, \ldots ,\varphi_n$
 be a partition of unity for $K$ with $\supp\varphi_l\subseteq V_{t_l}$ for $1\leq l\leq n$, and let
 $b_l:=\varphi_l\cdot b$ (pointwise product).
 Then $a:=\sum_{l=1}^n \delta^r(b_l)(1\otimes M(\tau_{t_l^{-1}}(g)f))\in \E_c$ such that
 \begin{align*}
&\left\|(1\otimes M(g))\delta^r(b)(1\otimes M(f))-\sum_{l=1}^n \delta^r(b_l)(1\otimes M(\tau_{t_l^{-1}}(g)f))\right\|\\
&=\left\|\int_K (\iota_\B(b_s)\otimes \lambda_s)(1\otimes M(\tau_{s^{-1}}(g)f))
-\sum_{l=1}^n (\iota_\B(\varphi_l(s)b_s)\otimes \lambda_s)(1\otimes M(\tau_{t_l^{-1}}(g)f))\,dt\right\|\\
 &\leq \int_K \sum_{i=1}^n \varphi_l(s) \|\iota_\B(b_s)\|\| \tau_{s^{-1}}(g)f-\tau_{t_l^{-1}}(g)f\|_\infty dt
 \leq \eps \mu(K)\|b\|_\infty,
 \end{align*}
 where $\mu(K)$ denotes the Haar measure of $K$. One checks as before that for any function $\varphi\in C_c(G)$
 with $\varphi\equiv 1$ on $\supp(f)\cup K\cdot\supp(f)$ we have $\varphi\cdot a=a\cdot\varphi=a$, which
 now shows that $\E_c$ is inductive limit dense in $A_c$.

 Recall now from \cite[Lemma 2.3]{Buss-Echterhoff:Exotic_GFPA} that there is a surjective linear map
 $\EE: A_c\to A_c^G$ given by the equation $\EE(a)c=\int_G\widehat{\delta}_t(a)c\,dt$ for all $a,c\in A_c$
 such that for all $m\in A_c^G$ and $f\in C_c(G)$ we have $\EE(m\cdot f)=\EE(m)\EE_\tau(f)$,
 with $\EE_\tau(f):=\int_G f(t)\,dt$. For $m=\delta^r(b)$ with $b\in C_c(\B)$ we get
 $m\cdot f=\delta^r(b)(1\otimes M(f))$ and it follows that $\EE(\E_c)=\delta^r(C_c(\B))$.
A slight adaptation of the last part of the proof of \cite[Lemma 2.7]{Buss-Echterhoff:Exotic_GFPA} shows that
$\EE: A_c\to A_c^G$ is continuous for the inductive limit topologies.
Hence, since $\E_c$ is inductive limit dense in $A_c$ it now follows that
$\delta^r(C_c(\B))=\EE(\E_c)$ is inductive limit dense in $A_c^G$, hence norm dense
in $A_u^G$.

Hence $\delta^r\colon \contc(\B)\to A^G_c$ extends to a surjective \Star{}homomorphism $\Psi\colon C^*(\B)\to A^G_u$.
We now check that $\Psi$ is equivariant with respect to the dual coaction on $C^*(\B)$ and the coaction $\delta_{A^G_u}$ on $A^G_u$ as defined in \cite{Buss-Echterhoff:Exotic_GFPA} on the dense subspace $A_c^G$ by the formula:
$$\delta_{A^G_u}(m)=(\phi\otimes\id)(w_G)(m\otimes 1)(\phi\otimes\id)(w_G)^*$$
where $w_G\in \M(\contz(G)\otimes C^*(G))$ is the unitary given by the function $t\mapsto u_t$ and
$\phi=1\otimes M: C_0(G)\to \M(A)$.
Recall that the equivariance of $\Psi$ means the following equality:
\begin{equation}\label{eq:equivariant-coactions}
\delta_{A^G_u}(\Psi(b))=(\Psi\otimes \id)(\delta_\B(b))\quad \forall b\in \contc(\B).
\end{equation}
Using $\Psi=\delta^r$ on $C_c(\B)$, the right-hand side is given by
$$(\Psi\otimes\id)(\delta_\B(b))=(\Psi\otimes\id)\left(\int_G\iota_\B(b_t)\otimes u_t\dd{t}\right)=
\int_G\iota_\B(b_t)\otimes\lambda_t\otimes u_t\dd{t}.$$
To compare this with the left hand side, observe that since $\phi=1\otimes M$, we have $(\phi\otimes \id)(w_G)=1\otimes \tilde w_G$, where $\tilde w_G\defeq (M\otimes \id)(w_G)\in\M(\K(L^2(G))\otimes C^*(G))=\Lb\big(L^2(G,C^*(G))\big)$ is the unitary given by the formula $\tilde w_G\zeta(t)=u_t\zeta(t)$ for all $\zeta\in \contc(G,C^*(G))\sbe L^2(G,C^*(G))$ (here we view $L^2(G,C^*(G))=L^2(G)\otimes C^*(G)$ as a Hilbert module over $C^*(G)$ and write $\Lb\big(L^2(G,C^*(G))\big)$ for the \cstar{}algebra of adjointable operators on it). It follows that
\begin{multline*}
\delta_{A^G_u}(\Psi(b))=(\phi\otimes \id)(w_G)(\Psi(b)\otimes 1)(\phi\otimes\id)(w_G)^*\\
=(1\otimes\tilde w_G)\left(\int_G \iota_\B(b_t)\otimes\lambda_t\otimes 1 \dd{t}\right)(1\otimes \tilde w_G)^*
=\int_G \iota_\B(b_t)\otimes \tilde w_G(\lambda_t\otimes 1)\tilde w_G^*\dd{t}.
\end{multline*}
Now a simple computation shows that $\tilde w_G(\lambda_t\otimes 1)\tilde w_G^*=\lambda_t\otimes u_t$,  which then implies Equation~\eqref{eq:equivariant-coactions}.

To finish the proof we only need to check that $\Psi$ induces an isomorphism $\Psi\rtimes\dualG\colon C^*(\B)\rtimes_{\delta_\B}\dualG\congto A^G_u\rtimes_{\delta_{A^G_u}}\dualG$. But for every coaction $\delta\colon B\to \M(B\otimes C^*(G))$ it is known that the image of $\delta^r\colon B\to \M(B\rtimes_\delta \dualG)$ is the reduced generalised fixed point algebra $A^G_r$ for the weak $G\rtimes G$-algebra $A=B\rtimes_\delta\dualG$ (endowed with the dual action and the canonical embedding $\contz(G)\to \M(A)$). A first reference for this fact is Quigg's
original version of Landstad duality for coactions (see \cite{Quigg:Landstad_duality}).
We have shown in \cite{Buss-Echterhoff:Exotic_GFPA} that $A^G_r$ carries a coaction $\delta_{A^G_r}$
given on $A_c^G$ by the same formula as $\delta_{A_u^G}$
and that $(A^G_r,\delta_{A^G_r})$ is the normalisation of $(B,\delta)$ where $\delta^r \colon B\to A^G_r$ serves as the normalisation map. This in particular means that $\delta^r$ induces an isomorphism $\delta^r\rtimes\dualG\colon B\rtimes_\delta\dualG\congto A^G_r\rtimes_{\delta_{A^G_r}}\dualG$.
Now it is clear that the map $\Psi\colon C^*(\B)\to A^G_u$ constructed above composed with the normalisation map
$\nu: A^G_u\to A^G_r$ (given by the identity map on $A^G_c$) is the canonical map $\delta^r\colon C^*(\B)\to A^G_r$.
Hence it follows that the composition of the following sequence of maps
$$C^*(\B)\rtimes_{\delta_\B} \widehat{G}\stackrel{\Psi\rtimes\widehat{G}}{\longrightarrow} A_u^G\rtimes_{\delta_{A_u^G}}\widehat{G}
\stackrel{\nu\rtimes\widehat{G}}{\longrightarrow} A_r^G\rtimes_{\delta_{A^G_r}}\widehat{G}$$
is an isomorphism. Since $(A^G_r,\delta_{A^G_r})$ is also a normalisation for $(A^G_u,\delta_{A^G_u})$ and hence $\nu\rtimes\widehat{G}\colon A_u^G\rtimes_{\delta_{A_u^G}}\widehat{G} \to A_r^G\rtimes_{\delta_{A^G_r}}\widehat{G}$ is also an isomorphism, this implies the desired isomorphism $\Psi\rtimes\dualG\colon C^*(\B)\rtimes_{\delta_\B}\dualG\congto A^G_u\rtimes_{\delta_{A^G_u}}\dualG$.
\end{proof}

\begin{remark}\label{rem:regular=normalisation}
The normalisation of $(C^*(\B),\delta_\B)$ can be realised concretely as the \emph{dual coaction} $\delta_{\B,r}\colon C^*_r(\B)\to \M(C^*_r(\B)\otimes C^*(G))$ of $G$ on $C^*_r(\B)$, which is constructed as follows. Consider the regular representation $\Lambda_\B\colon C^*(\B)\to C^*_r(\B)$ and view it as a representation $\Lambda_\B\colon \B\to \M(C^*_r(\B))$ of $\B$. Now consider the tensor product representation $\Lambda_\B\otimes\lambda\colon \B\to \M(C^*_r(\B)\otimes C^*_r(G))$. By Fell's absorbtion theorem \cite{ExelNg:ApproximationProperty}*{Corollary~2.15}, the integrated form of this representation factors faithfully through $C^*_r(\B)$ and hence yields a faithful \Star{}homomorphism $\Lambda_\B\otimes\lambda\colon C^*_r(\B)\to\M(C^*_r(\B)\otimes C^*_r(G))$. It is not difficult to check directly (see \cite{ExelNg:ApproximationProperty}*{Proposition~2.10} for details) that this is a reduced coaction (that is, an injective coaction of the Hopf-\cstar{}algebra $C^*_r(G)$), and therefore it lifts to a normal coaction $\delta_{\B,r}\colon C^*_r(\B)\to \M(C^*_r(\B)\otimes C^*(G))$. This is the desired normalisation of the dual coaction $\delta_\B\colon C^*(\B)\to \M(C^*(\B)\otimes C^*(G))$, with the regular representation $\Lambda_\B\colon C^*(\B)\to C^*_r(\B)$ serving as the normalisation map (see \cite{Buss:thesis}*{Proposition~6.9.8}).
\end{remark}

\section{Some applications}

In this section we want to give some simple applications of our main Theorem~\ref{theo:maximal-dual-coaction}.

\subsection{Extension of exotic crossed-product functors}

Recall from \cite{Buss-Echterhoff-Willett:Exotic,Baum-Guentner-Willett:Expanders} that an exotic crossed-product functor is a functor $(A,\alpha)\mapsto A\rtimes_{\alpha,\mu}G$ from the category of $G$-\cstar{}algebras with $G$-equivariant \Star{}homomorphisms to the category of \cstar{}algebras lying between the full and reduced crossed-product functors $A\rtimes_\alpha G$, $A\rtimes_{\alpha,r}G$. More concretely, this means that $A\rtimes_{\alpha,\mu}G$ is a \cstar{}completion of the convolution \Star{}algebra $\contc(G,A)$ in such a way that the identity map $\contc(G,A)\to \contc(G,A)$ extends to surjective \Star{}homomorphisms
$$A\rtimes_\alpha G\onto A\rtimes_{\alpha,\mu}G\onto A\rtimes_{\alpha,r}G.$$
Theorem~\ref{theo:maximal-dual-coaction} allows us to extend every Morita compatible $G$-crossed-product functor $\rtimes_\mu$ to the category of Fell bundles over $G$, that is, we can extend the definition of $\rtimes_\mu$ to the realm of Fell bundles over $G$ in a natural and functorial way. Recall from \cite{Buss-Echterhoff-Willett:Exotic} that a crossed product functor is called {\em Morita compatible}\footnote{Also called ``strongly Morita compatible'' in \cite{Buss-Echterhoff-Willett:Exotic} to differentiate it from the formally weaker (but essentially equivalent) notion of Morita compatibility introduced in \cite{Baum-Guentner-Willett:Expanders}.} if
Morita equivalent actions are sent to (canonically) Morita equivalent crossed products. We refer to \cite{Buss-Echterhoff-Willett:Exotic} for
a detailed discussion of this property and for the  stronger notion of a correspondence functor. As shown there, many crossed-product
functors do have this property, and it follows from work of Okayasu (\cite{Okayasu:Free-group}) together with the papers
 \cite{Kaliszewski-Landstad-Quigg:Exotic, Buss-Echterhoff-Willett:Exotic} that there are uncountably many different correspondence functors for any discrete group which contains the free group in two generators.

We shall show that starting with a 
crossed-product functor $(A,G,\alpha)\mapsto A\rtimes_{\alpha,\mu}G$, then for every Fell bundle $\B$ over $G$ we can complete $\contc(\B)$ to a \cstar{}algebra $C^*_\mu(\B)$ lying between $C^*(\B)$ and $C^*_r(\B)$ in the sense that the identity map on $\contc(\B)$ extends to surjections
$$C^*(\B)\onto C^*_\mu(\B)\onto C^*_r(\B)$$
and such that the assignment $\B\to C^*_\mu(\B)$ is a functor from the category of Fell bundles over $G$ (with appropriate morphisms) to the category of \cstar{}algebras with \Star{}homomorphisms as morphisms.
We make this precise in what follows.

\begin{definition}\label{def-functor}
Given a crossed-product functor $\rtimes_\mu$ for a locally compact group $G$ and a Fell bundle $\B$ over $G$, we define  $C^*_\mu(\B)$ as the unique quotient of $C^*(\B)$ such that Katayama's duality map
$$\Phi_\B: C^*(\B)\rtimes_{\delta_\B}\widehat{G}\rtimes_{\widehat{\delta_\B}}G\congto C^*(\B)\otimes \K(L^2(G))$$
factors through an isomorphism
$$C^*(\B)\rtimes_{\delta_\B}\widehat{G}\rtimes_{\widehat{\delta_\B},\mu}G\cong C_\mu^*(\B)\otimes \K(L^2(G)).$$
\end{definition}

Although the above construction makes sense for every crossed-product functor, as we will see, it will
only give a completion $C^*_\mu(\B)$ with good properties if we assume that the given  functor $\rtimes_\mu$
has extra properties (for instance, Morita compatibility). We are specially interested in correspondence functors,
where essentially all good properties are present (see \cite{Buss-Echterhoff-Willett:Exotic}).

To make the construction $\B\to C^*_\mu(\B)$ into a functor, we need to introduce morphisms and turn Fell bundles over $G$ into a category. As for \cstar{}algebras, there are several types of morphisms we can consider between Fell bundles, but the most basic one is defined as follows.

\begin{definition}\label{def-mor}
Let $\A$ and $\B$ be Fell bundles over $G$. By a morphism $\A\to\B$ we mean a continuous map $\pi\colon \A\to \B$ that maps each fibre $A_t$ linearly into the fibre $B_t$ and which is compatible with multiplication and involution in the sense that
$$\pi(a\cdot b)=\pi(a)\cdot \pi(b)\quad\mbox{and}\quad\pi(a)^*=\pi(a)^*$$
for all $a,b\in \A$.
\end{definition}

A morphism $\pi\colon \A\to \B$ induces a map $\tilde\pi\colon \contc(\A)\to \contc(\B)$, $\xi\mapsto \tilde\pi(\xi)(t)\defeq \pi(\xi(t))$, which is clearly continuous with respect to the inductive limit topologies and hence extends to a \Star{}homomorphism
$\tilde\pi_u\colon C^*(\A)\to C^*(\B)$.
This shows that the construction $\B\mapsto C^*(\B)$ is a functor. The following result shows that this remains true for
the assignment $\B\mapsto C_\mu^*(\B)$ as in Definition \ref{def-functor}:

\begin{proposition}\label{prop-mor}
 Let $\rtimes_\mu$ be any crossed product functor. Then
 $\B\mapsto C_\mu^*(\B)$ is a functor from the
category of Fell bundles with morphisms as defined in Definition \ref{def-mor} in the sense that the
canonical map $\tilde\pi: C_c(\A)\to C_c(\B)$ induced from any morphism $\pi:\A\to \B$ extends to
a $*$-homomorphism $\tilde\pi:C_\mu^*(\A)\to C_\mu^*(\B)$.
\end{proposition}
\begin{proof} Consider the diagram
$$
\begin{CD}
C^*(\A)\rtimes_{\delta_\A} \widehat{G}\rtimes_{\widehat{\delta_\A}}G  @>\Phi_\A>\cong> C^*(\A)\otimes \K(L^2(G))\\
@V\tilde{\pi}_u\rtimes\widehat{G}\rtimes G VV    @VV \tilde\pi_u\otimes\id_\K V\\
C^*(\B)\rtimes_{\delta_\B} \widehat{G}\rtimes_{\widehat{\delta_\B}}G  @>\Phi_\B>\cong> C^*(\B)\otimes \K(L^2(G)).
\end{CD}
$$
It follows easily from the definition of  the dual coactions on $C^*(\A)$ and $C^*(\B)$, respectively,
that the morphism $\tilde\pi_u:C^*(\A)\to C^*(\B)$ is $\delta_\A-\delta_\B$-equivariant, which implies that the
left vertical arrow exists. Moreover, using the fact that $\Phi_\A$ is given by
the covariant homomorphism $\big((\id\otimes\lambda)\circ \delta_\A\rtimes(1\otimes M)\big)\rtimes (1\otimes \rho)$
(and similarly for $\Phi_\B$), the $\delta_\A-\delta_\B$-equivariance of $\tilde\pi_u$ also implies that the diagram
commutes. Now, since $\rtimes_\mu$ is a crossed-product functor, the vertical arrow on the left factors through a
$*$-homomorphism
$$\tilde\pi_u\rtimes\widehat{G}\rtimes_\mu G:C^*(\A)\rtimes_{\delta_\A} \widehat{G}\rtimes_{\widehat{\delta_\A},\mu}G
\to C^*(\B)\rtimes_{\delta_\B} \widehat{G}\rtimes_{\widehat{\delta_\B},\mu}G $$
and hence  the vertical arrow $\tilde\pi_u\otimes\id_\K$ on the right-hand side of the diagram
must also factor through a well-defined homomorphism
$(\tilde\pi_u\otimes \id_\K)_\mu: C_\mu^*(\A)\otimes \K(L^2(G))\to C_\mu^*(\B)\otimes \K(L^2(G))$.
But this is only possible if $\tilde\pi_u:C^*(\A)\to C^*(\B)$ factors through a homomorphism
$\tilde\pi_\mu:C^*_\mu(\A)\to C^*_\mu(\B)$, whence the result.
\end{proof}

The above proposition shows that given any crossed-product functor, the procedure given in Definition \ref{def-functor}
determines a functor on the category of Fell bundles. But does it always extend the given functor if we
apply the new functor to the semi-direct product Fell bundle $A\rtimes^\alpha G$ associated to a given action
$\alpha:G\to \Aut(A)$? Recall that the underlying topological space of $A\rtimes^\alpha G$ is the
trivial bundle $A\times G$ with multiplication and involution defined by
$$(a, t)(b, s)=(a\alpha_t(b), ts)\quad\text{and}\quad (a,t)^*=(\alpha_{t^{-1}}(a)^*, t^{-1})$$
for $(a,t), (b,s)\in A\times G$.
The notation $A\rtimes^\alpha G$ for this Fell bundle should not be mistaken with the
notation for the universal crossed product $A\rtimes_{\alpha}G$.
The following example shows that the answer to the above question is negative in general:

\begin{example}\label{ex-functor}
Let $G$ be any non-amenable group. We define a crossed product functor $(A,G,\alpha)\mapsto A\rtimes_{\alpha,\mu}G$
by letting $A\rtimes_{\alpha,\mu}G$ be the completion of the convolution algebra $C_c(G,A)$ with respect to the
\cstar{}norm
$$\|f\|_\mu=\sup\{\|\pi\rtimes U(f)\| :  U\prec \lambda \},$$
where $(\pi,U)$ runs through all covariant representations such that
$U$ {\em is weakly contained in} $\lambda$, which just means that the
kernel of $U$ in $C^*(G)$ contains the kernel of $\lambda$ in $C^*(G)$. The functor we get in this way is just the
Brown-Guentner functor associated to the reduced group algebra $C_r^*(G)$ as discussed in \cite{BG, Baum-Guentner-Willett:Expanders, Buss-Echterhoff-Willett:Exotic}.

We now consider the case of the trivial action of $G$ on $\C$. Then the $\mu$-crossed product $\C\rtimes_\mu G$
will just be the reduced group algebra $C_r^*(G)$ of $G$. On the other hand, the
 corresponding Fell bundle will just be the trivial bundle $\C\times G$ and the full cross-sectional algebra
will be the full group algebra $C^*(G)$. The crossed product $C^*(G)\rtimes_{\delta_G}\widehat{G}$
by the dual coaction is isomorphic to the algebra of compact operators $\K(L^2(G))$ with
faithful representation $\lambda\rtimes M: C^*(G)\rtimes_{\delta_G}\widehat{G}\to \BB(L^2(G))$
(e.g., see \cite[Example A.62]{Echterhoff-Kaliszewski-Quigg-Raeburn:Categorical}).
A straightforward computation shows that in this picture the dual action
$\widehat{\delta}_G: G\to \Aut(\K(L^2(G)))$ is given by $\widehat\delta_G(s)=\Ad\rho(s)$, where
$\rho:G\to U(L^2(G))$ is the right regular representation.
Hence, this action is implemented by a unitary representation and there is a
$*$-isomorphism
$$\Phi: \K(L^2(G))\rtimes_{\Ad\rho}G\stackrel{\cong}{\longrightarrow} \K(L^2(G))\rtimes_{\id} G\cong \K(L^2(G))\otimes C^*(G)$$
which sends an element $f$ in the dense subalgebra $C_c(G,\K(L^2(G)))$ to the function
$[s\mapsto f(s)\rho(s)]\in C_c(G,\K(L^2(G)))$. The representations of $\K(L^2(G))\otimes C^*(G)$ are all of the form
$\id_{\K}\otimes V$, where  $V:G\to U(\H)$ is a unitary representation of $G$ viewed
 as a representation of $C^*(G)$ via integration.
 The corresponding covariant representation of $(\K(L^2(G)), G,\id)$ is given by the pair
 $(\id_{\K}\otimes 1_\H, 1_{L^2(G)}\otimes V)$ and it is easy to check that the representation
 of $\K(L^2(G))\rtimes_{\Ad\rho}G$ corresponding to  $(\id_{\K}\otimes 1_\H, 1_{L^2(G)}\otimes V)$
 via the isomorphism $\Phi$ is given by the covariant pair $(\id_{\K}\otimes 1, \rho\otimes V)$.
 By Fell's trick we know that $\rho\otimes V$ is a multiple of $\rho$. Since  $\rho$ is unitarily equivalent
 to $\lambda$, we see that the unitary part of any  representation of $\K(L^2(G))\rtimes_{\Ad\rho}G$
 is weakly equivalent  to $\lambda$. But this implies that
 $$C^*(G)\rtimes_{\delta_G}\widehat{G}\rtimes_{\widehat{\delta_G},\mu}G\cong
 \K(L^2(G))\rtimes_{\Ad\rho, \mu}G=\K(L^2(G))\rtimes_{\Ad\rho}G
 \cong \K(L^2(G))\otimes C^*(G).$$
 As a consequence we see that
 $$C_\mu^*(\C\times G)\cong C^*(G)\not\cong C_r^*(G)=\C\rtimes_{\mu}G.$$
 Thus, in general, our procedure does not reproduce the given functor for actions when applied to semi-direct product bundles
 $A\rtimes^\alpha G$.
 \end{example}

The problem in the above example comes from the fact that the Brown-Guentner  crossed-product functor
associated to the reduced group algebra $C_r^*(G)$  is not {\em Morita compatible} in the sense
discussed in \cite{Buss-Echterhoff-Willett:Exotic}. Recall from \cite{Buss-Echterhoff-Willett:Exotic} that a crossed-product functor is called Morita compatible if it preserves Morita equivalences in the following sense:
If $(X,\gamma)$ is a $G$-equivariant equivalence bimodule  between two systems
$(A,G,\alpha)$ and $(B,G,\beta)$, and if we make $C_c(G,X)$ into a $C_c(G,A)-C_c(G,B)$ pre-equivalence
bimodule in the usual way, then there is a completion $X\rtimes_{\gamma,\mu}G$ of $C_c(G,X)$
which becomes a $A\rtimes_{\alpha,\mu}G-B\rtimes_{\beta,\mu}G$ equivalence bimodule.
Note that the results of \cite{Buss-Echterhoff-Willett:Exotic} show that Morita compatibility -- or the even stronger assumption that
$\rtimes_\mu$ is a correspondence functor -- are quite reasonable to assume for a ``good behaved'' crossed-product functor.
Recall also from the discussions in \cite{Buss-Echterhoff-Willett:Exotic} that for many non-amenable groups there exist uncountably many
different correspondence (and hence Morita compatible) functors for $G$.

\begin{theorem}\label{thm-Morita}
Assume that $(A,G,\alpha)\mapsto A\rtimes_{\alpha,\mu}G$ is a Morita compatible crossed-product functor for $G$.
Then there is a canonical isomorphism $C_\mu^*(A\rtimes^\alpha G)\cong A\rtimes_{\alpha,\mu}G$ for any system
$(A,G,\alpha)$ given on the dense subalgebra $C_c(A\rtimes^\alpha G)$ by the
canonical identification $C_c(A\rtimes^\alpha G)=C_c(G,A)\subseteq A\rtimes_{\alpha,\mu}G$.
Thus, for Morita compatible crossed-product functors, the procedure of Definition \ref{def-functor}
defines an extension of the functor $\rtimes_\mu$ to the category of Fell bundles.
\end{theorem}
\begin{proof} The result is a consequence of the Imai-Takai duality theorem for actions: Assume that $\alpha:G\to\Aut(A)$
is an action. Then the full crossed product $A\rtimes_\alpha G$ coincides with the full cross-sectional algebra
$C^*(A\rtimes^\alpha G)$ since both are universal for covariant representations. Moreover, the dual coactions of
$G$ on $A\rtimes_{\alpha}G$ and $C^*(A\rtimes^\alpha G)$ coincide. Hence the Imai-Takai duality theorem shows that
$$C^*(A\rtimes^\alpha G)\rtimes_{\widehat{\alpha}}\widehat{G}=A\rtimes_{\alpha}G\rtimes_{\widehat{\alpha}}\widehat{G}\cong A\otimes \K(L^2(G)).$$
which is equivariant for the bi-dual action $\widehat{\widehat{\alpha\,}}$ on the left and the action
$\alpha\otimes \Ad\rho$ on the right (e.g., see \cite{IT} or
\cite[Theorem A.67]{Echterhoff-Kaliszewski-Quigg-Raeburn:Categorical}). As already observed in the previous example, the action $\alpha\otimes \Ad\rho$ is
Morita equivalent to $\alpha\otimes \id$ with respect to the equivariant  bimodule $(A\otimes \K(L^2(G)), \alpha\otimes \rho)$.
By Corollary~5.4 in \cite{Buss-Echterhoff-Willett:Exotic},  it follows that the integrated form $\Psi_\rho$ of the covariant homomorphism
$(i_A\otimes \id_\K, i_G\otimes \rho)$ factors through an isomorphism
$$\Psi_\mu :\big(A\otimes \K(L^2(G))\big)\rtimes_{\alpha\otimes\Ad\rho,\mu}G\stackrel{\cong}{\longrightarrow}
(A\rtimes_{\alpha,\mu}G)\otimes \K(L^2(G)),$$
where $(i_A^\mu,i_G^\mu)$ denotes the canonical representation of $(A,G,\alpha)$ into $\M(A\rtimes_{\alpha,\mu} G)$.
Combined, we obtain an isomorphism
$$C^*(A\rtimes^\alpha G)\rtimes_{\widehat{\alpha}}\widehat{G}\rtimes_{\widehat{\widehat{\alpha\,}},\mu}G\cong  (A\rtimes_{\alpha,\mu}G)\otimes \K(L^2(G))$$
which fits into a commutative diagram
$$
\begin{CD}
C^*(A\rtimes^\alpha G)\rtimes_{\widehat{\alpha}}\widehat{G}\rtimes_{\widehat{\widehat{\alpha\,}}}G
@>\Psi_{u}>>  (A\rtimes_{\alpha}G)\otimes \K(L^2(G))\\
@VVV   @VVV\\
C^*(A\rtimes^\alpha G)\rtimes_{\widehat{\alpha}}\widehat{G}\rtimes_{\widehat{\widehat{\alpha\,}},\mu}G
@>>\Psi_\mu> (A\rtimes_{\alpha,\mu}G)\otimes \K(L^2(G))
\end{CD}
$$
where both vertical arrows are induced by the natural inclusion $C_c(A\rtimes^\alpha G)=C_c(G,A)\subseteq A\rtimes_{\alpha,\mu}G$.
This finishes the proof.
\end{proof}

In particular the above result applies to \emph{correspondence} crossed-product functors as introduced in \cite{Buss-Echterhoff-Willett:Exotic}.
This is a class of crossed-product functors which extend (in a suitable way) to the category of $G$-algebras with equivariant correspondences as their morphisms. The equivariant Morita equivalences are the isomorphisms in this category, so it is no surprise that these functors are Morita compatible. In \cite{Buss-Echterhoff-Willett:Exotic}*{Theorem~4.9} a list of equivalent conditions is given in order to check whether a crossed-product functor is a correspondence functor.
It is shown in \cite{Buss-Echterhoff-Willett:Exotic} that correspondence functors have very nice properties. For instance, they behave very well with respect to $K$-theory, and we will explore this point in the next section. Here we want to use the fact, proven in \cite{Buss-Echterhoff-Willett:Exotic}*{Theorem~5.6}, that the correspondence functors always admit dual coactions for ordinary crossed products and deduce the following consequence:

\begin{corollary}
Let $\rtimes_\mu$ be a correspondence crossed-product functor for $G$ and let $\B$ be a Fell bundle over $G$.
Then the dual coaction $\delta_\B$ on $C^*(\B)$ factors through a coaction $\delta_{\B,\mu}\colon C^*_\mu(\B)\to \M(C^*_\mu(\B)\otimes C^*(G))$, which we call the \emph{dual $\mu$-coaction}. The quotient maps $C^*(\B)\onto C^*_\mu(\B)\onto C^*_r(\B)$ induce isomorphisms
\begin{equation}\label{eq:all-crossed-product-isomorphic}
C^*(\B)\rtimes_{\delta_\B}\dualG\congto C^*_\mu(\B)\rtimes_{\delta_{\B,\mu}}\dualG\congto C^*_\red(\B)\rtimes_{\delta_\B^r}\dualG.
\end{equation}
Hence the dual $\mu$-coaction satisfies \emph{$\mu$-duality} in the sense that Katayama's map is an isomorphism
\begin{equation}\label{eq:mu-coaction-satisfy-mu-duality}
C^*_\mu(\B)\rtimes_{\delta_{\B,\mu}}\dualG\rtimes_{\dual\delta_{\B,\mu}}G\congto C^*_\mu(\B)\otimes \K(L^2(G)).
\end{equation}
This isomorphism sends the bidual coaction $\dual{\dual\delta}_{\B,\mu}$ to the coaction $\Ad_{W}\circ (\delta_{\B,\mu}\otimes_*\id)$,
where $W=1\otimes w_G^*$, $w_G\in \M(\contz(G)\otimes C^*(G))$ is the fundamental unitary (which can be seen as the universal representation $t\mapsto u_t$ of $G$), and $\delta_{\B,\mu}\otimes_*\id$ denotes the obvious coaction $C^*_{{\mu}}(\B)\otimes\K(L^2(G))\to \M(C^*_{{\mu}}(\B)\otimes \K(L^2(G))\otimes C^*(G))$.
\end{corollary}
\begin{proof}
By Theorem~\ref{theo:maximal-dual-coaction}, Katayama's homomorphism is an isomorphism
\begin{equation}\label{Katayama-Isomorphism}
C^*(\B)\rtimes_{\delta_{\B}}\dualG\rtimes_{\dual\delta_{\B}}G\congto C^*(\B)\otimes \K(L^2(G)).
\end{equation}
It is well known (see e.g. \cite{Echterhoff-Kaliszewski-Quigg:Maximal_Coactions}) that the bidual coaction $\dual{\dual\delta}_\B$ on the left-hand side corresponds to the coaction $\Ad_{W}\circ (\delta_{\B}\otimes_*\id)$ as in the statement.
By definition, $C^*_\mu(\B)$ is the quotient of $C^*(\B)$ that turns~\eqref{Katayama-Isomorphism} into an isomorphism
\begin{equation}\label{Katayama-Isomorphism-mu}
C^*(\B)\rtimes_{\delta_{\B}}\dualG\rtimes_{\dual\delta_{\B,\mu}}G\congto C^*_\mu(\B)\otimes \K(L^2(G)).
\end{equation}
Since $\rtimes_\mu$ is a correspondence functor, the left-hand side carries a (bi)dual coaction $\dual{\dual\delta}_{\B,\mu}$ (by \cite{Buss-Echterhoff-Willett:Exotic}*{Theorem~5.6}). More precisely, the bidual coaction on the full crossed product $C^*(\B)\rtimes_{\delta_{\B}}\dualG\rtimes_{\dual\delta_{\B}}G$ factors through the coaction $\dual{\dual\delta}_{\B,\mu}$. It follows that the coaction $\Ad_{W}\circ (\delta_{\B}\otimes_*\id)$ also factors through a coaction on $C^*_\mu(\B)\otimes \K(L^2(G))$ of the form $\Ad_{W}\circ (\delta_{\B,\mu}\otimes_*\id)$, where $\delta_{\B,\mu}$ is a coaction of $C^*_\mu(\B)$ which factors the dual coaction $\delta_\B$ on $C^*(\B)$. This holds in particular for the reduced cross-sectional algebra $C^*_r(\B)$ and, as already observed in Remark~\ref{rem:regular=normalisation}, in this case the coaction $\delta_{\B,r}$ is a normalisation of $\delta_\B$.
In particular the quotient homomorphism $C^*(\B)\onto C^*_r(\B)$ (which is the regular representation of $\B$) induces an isomorphism $C^*(\B)\rtimes_{\delta_\B}\dualG\congto C^*_r(\B)\rtimes_{\delta_{\B,r}}\dualG$.
It follows that the same is also true for every other exotic quotient $C^*_\mu(\B)$ because the quotient map $C^*(\B)\onto C^*_r(\B)$ (and hence also the induced map on crossed products) factors as a composition $C^*(\B)\onto C^*_\mu(\B)\onto C^*_r(\B)$. This implies the isomorphism~\eqref{eq:all-crossed-product-isomorphic} and the isomorphism~\eqref{eq:mu-coaction-satisfy-mu-duality} is then just a re-interpretation of the defining isomorphism~\eqref{Katayama-Isomorphism-mu}.
\end{proof}

\subsection{$K$-amenability}

The concept of $K$-amenable groups has first been introduced for discrete groups
by Cuntz in
\cite{Cuntz:K-amenable}
and has then been extended to locally compact groups
by Julg and Valette in \cite{JV}. It follows from the results of Cuntz that a
 discrete group $G$ is $K$-amenable if and only if the regular representation
  $\lambda: C^*(G)\to C_r^*(G)$ is a $KK$\nb-equivalence, which then implies that for
  all actions $\alpha:G\to\Aut(A)$ the regular representation $\Lambda_A^G: A\rtimes_{\alpha}G\to A\rtimes_{\alpha,r}G$
  is a $KK$\nb-equivalence as well.
   The definition of $K$-amenability for general  locally compact groups
  is slightly more technical, but as a consequence we also get that  regular representations of
  crossed-products induce $KK$\nb-equivalences between the full and reduced crossed products.
  More generally, it is shown in \cite[Theorem 5.6]{Buss-Echterhoff-Willett:Exotic} that, if
  $G$ is $K$-amenable, then
  for any correspondence functor $\rtimes_\mu$, the canonical quotient maps
  $$A\rtimes_{\alpha}G\onto A\rtimes_{\alpha,\mu}G\onto A\rtimes_{\alpha,r}G$$
  are $KK$\nb-equivalences.
  Cuntz has shown in \cite{Cuntz:K-amenable} that all free groups are $K$-amenable and that
  $K$-amenability enjoys some nice permanence properties. Moreover, a more recent result of Tu
  \cite{Tu} shows  that all a-$T$-menable groups are $K$-amenable as well.

The results of the previous section now allow us to extend  \cite[Theorem 5.6]{Buss-Echterhoff-Willett:Exotic} to
cross-sectional algebras of Fell bundles. For general Fell bundles, the result seems to be new even
for the quotient map $\Lambda_\B: C^*(\B)\onto C_r^*(\B)$, but this special case is known for Fell bundles associated to partial actions of discrete groups (see \cite{McClanahan:K-theory-partial} and Section~\ref{Sec:Partial-Actions} below for further discussion).

\begin{corollary}\label{cor:K-amenable->KK-equivalent}
Let $G$ be a $K$-amenable locally compact group and let $\rtimes_\mu$ be a correspondence crossed-product functor for $G$.
Then both $*$-homomorphisms in the  sequence
$$C^*(\B)\onto C_\mu^*(\B)\onto C_r^*(\B)$$
given by the identity maps on $C_c(\B)$ are $KK$\nb-equivalences. In particular, they induce
isomorphisms $K_*(C^*(\B))\cong K_*(C_\mu^*(\B))\cong K_*(C_r^*(\B))$.
\end{corollary}
\begin{proof}
Consider the following commutative diagram
$$
\begin{CD}
C^*(\B)\rtimes_{\delta_\B}\widehat{G}\rtimes_{\widehat{\delta_G}}G   @>\Psi_\B >\cong >  C^*(\B)\otimes \K(L^2(G))\\
@VVV   @VVV\\
C^*(\B)\rtimes_{\delta_\B}\widehat{G}\rtimes_{\widehat{\delta_G},\mu}G   @>\Psi_\mu >\cong >  C_\mu^*(\B)\otimes \K(L^2(G))\\
@VVV   @VVV\\
C^*(\B)\rtimes_{\delta_\B}\widehat{G}\rtimes_{\widehat{\delta_G},r}G   @>\Psi_r >\cong >  C_r^*(\B)\otimes \K(L^2(G)).
\end{CD}
$$
Since $G$ is $K$-amenable, the vertical arrows on the left hand side are $KK$\nb-equivalences. Hence the right
vertical arrows are $KK$\nb-equivalences as well. Since being $KK$\nb-equi\-va\-lence is stable under stabilisation
by compact operators, the result follows.
\end{proof}

\begin{remark} The above result can be generalised as follows. Let $\rtimes_{\mu}$ be any crossed-product functor. Following
\cite{Buss-Echterhoff:Exotic_GFPA} we say that a given coaction $\delta:B\to \M(B\otimes C^*(G))$
is a $\mu$-coaction if Katayama's duality surjection
$B\rtimes_\delta \widehat{G}\rtimes_{\widehat{\delta}}G\onto B\otimes \K(L^2(G))$
factors through an isomorphism
$B\rtimes_\delta \widehat{G}\rtimes_{\widehat{\delta},\mu}G\cong B\otimes \K(L^2(G))$.
In particular, if $\rtimes_\mu$ is a correspondence functor and $\B$ is a Fell bundle over $G$, then
the dual coaction $\delta_{\B,\mu}$ on $C_\mu^*(\B)$ is a $\mu$-coaction.

Suppose now that $G$ is $K$-amenable, $\rtimes_\mu$ is a correspondence functor for $G$ and $(B,\delta)$ is a $\mu$-coaction. Then, if
$(B_m,\delta_m)$ and $(B_n,\delta_n)$ are the maximalisation and normalisation
of $(B,\delta)$, respectively, the corresponding quotient maps
$$B_m\stackrel{q_m}{\onto} B\stackrel{q_n}{\onto}B_n$$
are $KK$\nb-equivalences. This follows directly from the commutative
diagram
$$
\begin{CD}
B_m\rtimes_{\delta_m} \widehat{G}\rtimes_{\widehat{\delta},\mu}G @>\cong >> B_m\otimes \K(L^2(G))\\
@Vq_m\rtimes\widehat{G}\rtimes GVV   @VVq_m\otimes \id_\K V\\
B\rtimes_\delta \widehat{G}\rtimes_{\widehat{\delta},\mu}G @>\cong >>  B\otimes \K(L^2(G))\\
@Vq_n\rtimes\widehat{G}\rtimes GVV   @VVq_n\otimes \id_\K V\\
B_n\rtimes_{\delta_n} \widehat{G}\rtimes_{\widehat{\delta},\mu}G @>\cong >>  B_n\otimes \K(L^2(G))
\end{CD}
$$
and the fact that both morphisms in the sequence
$$B_m\rtimes_{\delta_m} \widehat{G}\stackrel{q_m\rtimes \widehat{G}}{\longrightarrow}
B\rtimes_{\delta} \widehat{G}\stackrel{q_n\rtimes \widehat{G}}{\longrightarrow}
B_n\rtimes_{\delta_n} \widehat{G}$$
are $G$-equivariant isomorphisms. In particular, if $G$ is $K$-amenable, all \cstar{}algebras $B_m$, $B$ and $B_n$ have same $K$-theory and $K$-homology groups.
\end{remark}

\subsection{Partial actions}\label{Sec:Partial-Actions}
The notion of partial actions of the group of integers has been introduced by Exel in \cite{Exel:Circle} and subsequently generalized to arbitrary discrete groups by McClanahan in \cite{McClanahan:K-theory-partial}. In \cite{Exel:TwistedPartialActions} Exel generalizes both notions and defines twisted partial actions of locally compact groups. Every twisted partial action gives rise to a Fell bundle via a construction similar to the semidirect Fell bundle associated to an ordinary (global, untwisted) action. Moreover, the main result in \cite{Exel:TwistedPartialActions} shows that, after stabilisation, every Fell bundle is isomorphic to one of this form, that is, a Fell bundle associated to a twisted partial action (and for discrete groups or saturated Fell bundles, the twist can be removed; see \cite{Exel:Book, Sehnem:MasterThesis, Packer-Raeburn:Stabilisation, Echterhoff:Morita_twisted}). In this section, we will focus only on partial actions, but essentially all results go through with essentially no change (except that the notation becomes slightly more complicated) for general twisted partial actions.

Let $\alpha$ be a partial action of a locally compact group $G$ on a \cstar{}algebra $A$. This consists of partial automorphisms $\alpha_t\colon D_{t^{-1}}\to D_t$ between certain (closed, two-sided) ideals $D_t\sbe A$ with $D_e=A$, $\alpha_e=\id_A$ and such that $\alpha_{st}$ extends $\alpha_s\circ\alpha_t$ for all $s,t\in G$. An appropriate continuity condition for the family of maps $\alpha_t$ is also required to hold. We refer the reader to \cite{Exel:TwistedPartialActions} for details. Given such a partial action, the associated Fell bundle is the bundle $A\rtimes^\alpha G\defeq \{(a,t)\in A\times G: a\in D_t\}$ with algebraic operations defined by
$$(a,t)\cdot (b,s)=(\alpha_t(\alpha_{t^{-1}}(a)b),ts)\quad\mbox{and}\quad (a,t)^*=(\alpha_{t^{-1}}(a)^*,t^{-1}).$$
The full (resp. reduced) crossed product of $(A,\alpha)$ can defined as the full (resp. reduced) cross-sectional \cstar{}algebras of $A\rtimes^\alpha G$. More generally, we can introduce:

\begin{definition}
Given a partial action $(A,\alpha)$ of $G$ and a Morita compatible $G$\nb-crossed-product functor $\rtimes_\mu$, we define the $\mu$-crossed product $A\rtimes_{\alpha,\mu}G$ as the $\mu$-cross-sectional \cstar{}algebra $C^*_\mu(A\rtimes^\alpha G)$ (as in Definition~\ref{def-functor}).
\end{definition}

Notice that by Theorem~\ref{thm-Morita} (and our assumption of Morita compatibility), the above definition recovers the original $\mu$-crossed product for global actions. Also, Proposition~\ref{prop-mor} allows us to extend the original crossed-product functor to a functor $(A,\alpha)\mapsto A\rtimes_{\alpha,\mu}G$ from the category of partial $G$-actions to the category of \cstar{}algebras, where a morphism between two partial $G$-actions $(A,\alpha)$ and $(B,\beta)$ is defined as $G$-equivariant \Star{}homomorphism $\pi\colon A\to B$, meaning that $\pi(D_t^\alpha)\sbe D_t^\beta$ and $\beta_t(\pi(a))=\pi(\alpha_t(a))$ for all $t\in G$ and $a\in D_{t^{-1}}^\alpha$. These are exactly the morphisms between the associated semidirect Fell bundles $A\rtimes^\alpha G$ and $B\rtimes^\beta G$.

Given a partial action $(A,\alpha)$, we view the dual coaction on $C^*(A\rtimes^\alpha G)$ as \emph{the dual coaction} of $A\rtimes_\alpha G$ and denote it by $\dual\alpha$. Our main result (Theorem~\ref{theo:maximal-dual-coaction}) says that we have a natural isomorphism
\begin{equation}\label{eq:double-dual-partial}
(A\rtimes_\alpha G)\rtimes_{\dual\alpha}\dualG\rtimes_{\dual{\dual\alpha}}G\cong (A\rtimes_\alpha G)\otimes\K(L^2(G)).
\end{equation}
In particular this implies that, after stabilisation, every partial crossed product is isomorphic to a global crossed product. More precisely, the (stabilisation of the) original partial crossed product $(A\rtimes_\alpha G)\otimes\K(L^2(G))$ is naturally isomorphic to $\tilde A \rtimes_{\tilde\alpha} G$, where $\tilde A\defeq(A\rtimes_\alpha G)\rtimes_{\dual\alpha}\dualG$ and $\tilde\alpha\defeq\dual{\dual\alpha}$. The $G$-algebra $(\tilde A,\tilde\alpha)$ has a natural interpretation in terms of the original partial action: it follows from \cite{Abadie:Enveloping}*{Proposition~8.1} that $(\tilde A,\tilde\alpha)$ is a \emph{Morita enveloping action} for the partial action $(A,\alpha)$; we call $(\tilde A,\tilde\alpha)$ the \emph{canonical Morita enveloping action} of $(A,\alpha)$. It is shown in \cite{Abadie:Enveloping}*{{Proposition~6.3}} that all Morita enveloping actions are Morita equivalent, so it is unique up to Morita equivalence. We refer to \cite{Abadie:Enveloping} for
the relevant notion of Morita equivalence for partial actions.

Let us recall that the assertion that $(\tilde A,\tilde\alpha)$ is a Morita enveloping action of $(A,\alpha)$ means that $\tilde A$ contains a (closed, two-sided) ideal $I\sbe \tilde A$ such that the orbit $\{\tilde\alpha_t(I):t\in G\}$ of $I$ generates a dense subspace of $\tilde A$ and such that the partial action on $I$ given by restriction of $\tilde\alpha$ (as described in \cite{Abadie:Enveloping}) is Morita equivalent to the original partial action $(A,\alpha)$. For the canonical Morita enveloping action, the ideal $I$ and also $\tilde A$ (together with their actions) can be described directly in terms of the Fell bundle $A\rtimes^\alpha G$ associated to $(A,\alpha)$ as certain algebras of ``kernels'', but this description does not concern us here. We refer to \cite{Abadie:Enveloping} for further details.

Thus the natural isomorphism~\eqref{eq:double-dual-partial} (that is, our main Theorem~\ref{theo:maximal-dual-coaction} applied for partial actions) can be seen as the statement that every partial crossed product $A\rtimes_\alpha G$ is naturally Morita equivalent to its canonical Morita enveloping crossed product. As already mentioned above, it is proven in \cite{Abadie:Enveloping} that all Morita enveloping actions of a given partial action $(A,\alpha)$ are Morita equivalent. It follows that the full
crossed product $A\rtimes_{\alpha}G$  is Morita equivalent to the full crossed product of any of the Morita enveloping actions of
$(A,\alpha)$.
Note that the paper \cite{Abadie:Enveloping} by Abadie  only contained a version of this result for the reduced crossed products (see
 \cite{Abadie:Enveloping}*{Proposition~4.6}).  The version for full crossed products was  obtained in a second paper \cite{Abadie-Perez:Amenability}*{Corollary~1.3} by Fernando Abadie in colaboration with Laura P\'erez.
We now use our approach to generalise these results
to all exotic crossed products related to Morita compatible crossed-product functors:

\begin{corollary}
For every partial action $(A,\alpha)$ of $G$, and for every Morita compatible crossed-product functor $\rtimes_\mu$,
the partial crossed product $A\rtimes_{\alpha,\mu}G$ is stably isomorphic to its canonical Morita enveloping crossed product $\tilde A\rtimes_{\alpha,\mu}G$. More precisely, the canonical isomorphism~\eqref{eq:double-dual-partial} factors through an isomorphism
\begin{equation}\label{eq:double-dual-partial-exotic}
\tilde A\rtimes_{\tilde\alpha,\mu}G\cong (A\rtimes_{\alpha,\mu} G)\otimes\K(L^2(G)).
\end{equation}
More generally, every other Morita enveloping action $(B,\beta)$ of $(A,\alpha)$ has crossed product $B\rtimes_{\beta,\mu}G\sim_M A\rtimes_{\alpha,\mu}G$.
\end{corollary}
\begin{proof}
The first statement is an immediate consequence of our definitions. Indeed, by definition,
the exotic partial crossed product $A\rtimes_{\alpha,\mu}G$ is exactly the quotient of $A\rtimes_{\alpha}G$ that turns~\eqref{eq:double-dual-partial} into the isomorphism~\eqref{eq:double-dual-partial-exotic}. And the final assertion follows from the already mentioned fact that all Morita enveloping actions are Morita equivalent and the assumption that our functor $\rtimes_\mu$ preserves Morita equivalence. Indeed, if $(B,\beta)$ is a Morita enveloping action for $(A,\alpha)$, then $(B,\beta)$ is Morita equivalent to $(\tilde A,\tilde\alpha)$ by \cite{Abadie:Enveloping}*{Proposition~6.3}. Since our crossed-product functor $\rtimes_\mu$ preserves Morita equivalences, we conclude that $B\rtimes_{\beta,\mu}G\sim_M \tilde A\rtimes_{\tilde\alpha,\mu}G\sim_M A\rtimes_{\alpha,\mu}G$.
\end{proof}

We also obtain one of the main results on amenability of partial actions shown in \cite{Abadie-Perez:Amenability}.
Following the terminology from \cite{Abadie-Perez:Amenability}, we say that a partial action $(A,\alpha)$ is \emph{amenable} if its associated Fell bundle is amenable (in the sense of Exel \cite{Exel:Amenability}). Hence, by definition, a partial action $(A,\alpha)$ is amenable if and only if its full and reduced crossed products coincide.

\begin{corollary}
A partial action $(A,\alpha)$ is amenable if and only if its canonical Morita enveloping action $(\tilde A,\tilde\alpha)$ is amenable, if and only if all Morita enveloping actions of $(A,\alpha)$ are amenable. In this case all exotic (partial) crossed products involving these algebras coincide. More generally, given Morita compatible $G$-crossed product functors $\rtimes_\mu$ and $\rtimes_\nu$, we have $A\rtimes_{\alpha,\mu}G= A\rtimes_{\alpha,\nu}G$ if and only if $\tilde A\rtimes_{\tilde\alpha,\mu}G=\tilde A\rtimes_{\tilde\alpha,\nu}G$ if and only if $B\rtimes_{\beta,\mu}G=B\rtimes_{\beta,\nu}G$ for every Morita enveloping action $(B,\beta)$ of $(A,\alpha)$.
\end{corollary}
\begin{proof}
The first assertion will follow from the last assertion by taking the full and reduced crossed products for $\rtimes_\mu$ and $\rtimes_\nu$.
To prove the last assertion notice that (by definition) the equality $A\rtimes_{\alpha,\mu}G=A\rtimes_{\alpha,\nu}G$ means that the ideal in the full crossed product $A\rtimes_\alpha G$ corresponding to the quotient maps $A\rtimes_\alpha G\to A\rtimes_{\alpha,\mu}G, A\rtimes_{\alpha,\nu}G$ coincide, and of course the same meaning is to be given for the equality $B\rtimes_{\beta,\mu}G=B\rtimes_{\beta,\nu}G$. But then the last assertion in the statement follows from the Rieffel correspondence between ideals induced by the Morita equivalences $A\rtimes_{\alpha,\mu}G\sim_M B\rtimes_{\beta,\mu}G$ and $A\rtimes_{\alpha,\nu}G\sim_M B\rtimes_{\beta,\nu}G$ and the fact that both are quotients of the Morita equivalence for full crossed products: $A\rtimes_\alpha G\sim_M B\rtimes_\beta G$.
\end{proof}

As a direct consequence of our Corollary~\ref{cor:K-amenable->KK-equivalent}, we also obtain the following result:

\begin{corollary}
Let $(A,\alpha)$ be a partial action of a locally compact $K$-amenable group.
If $\rtimes_\mu$ is a correspondence $G$-crossed-product functor, then the quotient homomorphism $A\rtimes_{\alpha}G\onto A\rtimes_{\alpha,\mu} G$ is a $KK$\nb-equivalence. In particular, $A\rtimes_{\alpha,\mu}G$ has the same $K$-theory and $K$-homology as $A\rtimes_\alpha G$.
\end{corollary}

For partial actions of \emph{discrete} groups, the above result was proven by McClanahan in \cite{McClanahan:K-theory-partial} for the special case of the quotient map $A\rtimes_{\alpha}G\onto A\rtimes_{\alpha,r}G$ linking the full and reduced crossed products by partial actions. Notice that the result of  McClanahan does not imply the result for general exotic crossed products for discrete groups.  Indeed, as shown in \cite{Buss-Echterhoff-Willett:Exotic}, there are examples of crossed-product functors that are
not correspondence functors for which the above result fails even for crossed products by ordinary actions.
 In the recent paper by Ara and Exel \cite{Ara-Exel:K-theory} (see in particular Corollary~6.9) the authors have applied the result by McClanahan to some interesting partial actions of free groups associated to separated graphs in order to deduce that certain full and reduced crossed products have the same $K$-theory (and the $K$-theory is effectively computed in \cite{Ara-Exel:K-theory}).
 By the above result these computations extend to the respective exotic crossed product related to any given correspondence crossed-product functor.

\section{Maximal coactions of discrete groups}

As a bonus, we derive in this section a characterisation of maximal coactions of discrete groups.
Recall from \cite{Ng:Discrete-coactions,Quigg:Discrete-coactions} that every coaction $\delta:B\to B\otimes C^*(G)$ of a discrete group $G$
determines a Fell bundle $\B$ over $G$ with fibres
$$B_t = \{b\in B:\delta(b)=b\otimes u_t\}$$
where $u:G\to U(C^*(G))$ denotes the inclusion map. There is a canonical embedding
$$C_c(\B)=\spn\left(\cup_{t\in G} B_t\right)\into B$$
which then extends to a surjective $\delta_\B-\delta$-equivariant
$*$-homomorphism $\kappa:C^*(\B)\onto B$.
On the other hand, it has also been observed by Quigg that the dual coaction $\delta_n$ of $G$ on $C_r^*(\B)$
is the normalisation of $(B,\delta)$, so that there is also a $\delta-\delta_n$-equivariant $*$-homomorphism
$\Lambda_B: B\mapsto B_r:=C_r^*(\B)$. This shows that, in a sense, we may view $B$ as an exotic completion of $C_c(\B)$.
We know from \cite{Echterhoff-Kaliszewski-Quigg:Maximal_Coactions} (and now also from our main Theorem~\ref{theo:maximal-dual-coaction}) that $(C^*(\B),\delta_\B)$ is the maximalisation of $(B,\delta)$, so that the quotient
maps $C^*(\B)\stackrel{\kappa}{\longrightarrow} B\stackrel{\Lambda_B}{\longrightarrow} C_r^*(\B)$ induce
isomorphisms
$$C^*(\B)\rtimes_{\delta_\B}\widehat{G}\congto B\rtimes_{\delta}\widehat{G}\congto C_r^*(\B)\rtimes_{\delta_n}\widehat{G}.$$

\begin{theorem}\label{theo:maximal-discrete}
Let $G$ be a discrete group and let $\delta\colon B\to B\otimes C^*(G)$ be a coaction.
Let $\B$ be the Fell bundle over $G$ corresponding to $\delta$ as explained above.
Then the following assertions are equivalent.
\begin{enumerate}
\item $(B,\delta)$ is maximal;
\item the canonical $*$-homomorphism $C^*(\B)\onto B$ is an isomorphism;
\item $\delta$ can be lifted to a \Star{}homomorphism $\delta_\max\colon B\to B\otimes_{\max}C^*(G)$;
\item the reduction $\delta^r\defeq (\id\otimes\lambda)\circ \delta\colon B\to B\otimes C^*_\red(G)$ can be lifted to a \Star{}homomorphism $\delta_\max^r\colon B\to B\otimes_{\max}C^*_r(G)$;
\item the \Star{}homomorphism $(\Lambda_B\otimes\id)\circ \delta^r = (\Lambda_B\otimes\lambda)\circ \delta\colon B\to B_r\otimes C^*_\red(G)$
can be lifted to a  homomorphism $(\Lambda_B\otimes_\max\id)\circ\delta_\max^r\colon B\to B_r\otimes_\max C^*_\red(G)$.
\end{enumerate}
Moreover, if they exist, the lifted homomorphisms in (3), (4), and (5) are all faithful.
\end{theorem}
\begin{proof}
The equivalence (1)$\Leftrightarrow$(2) follows from the fact that $(C^*(\B),\delta_\B)$ is the
maximisation of $(B,\delta)$ by \cite[Proposition~4.2]{Echterhoff-Kaliszewski-Quigg:Maximal_Coactions}.
Notice that (3) holds for  $(B,\delta)=(C^*(\B),\delta_\B)$ because of the universal property of $C^*(\B)$, so we also get (2)$\Rightarrow$(3), and it is obvious that (3)$\Rightarrow$(4). We will finish with the proof of the implications (4)$\Rightarrow$(5)$\Rightarrow$(2). For this assume that (4) holds. It is shown in \cite{Ara-Exel-Katsura:Dynamical_systems}*{Theorem~6.2} that for every Fell bundle $\B$ over a discrete group $G$, the representation of $\B$ into $C^*_r(\B)\otimes_\max C^*_\red(G)$ given by $b_t\to b_t\otimes_\max \lambda_t$ extends to a faithful \Star{}homomorphism $(\Lambda_{\B}\otimes_\max\lambda)\circ \delta_\B \colon C^*(\B)\to C^*_r(\B)\otimes_\max C^*_r(G)$
(and this is exactly the lift homomorphism in (5) for the  coaction $\delta_\B\colon C^*(\B)\to C^*(\B)\otimes C^*(G)$).
Now notice that the homomorphism $\kappa\colon C^*(\B)\to B$ fits into the commutative diagram:
\begin{equation}\label{eq-diag}
\begin{CD}
B @>\delta_\max^r >> B\otimes_\max C^*_r(G)\\
@A\kappa AA  @VV\Lambda_B\otimes_\max\id V\\
C^*(\B) @> (\Lambda_{\B}\otimes_\max\lambda)\circ \delta_\B >> C^*_r(\B)\otimes_\max C^*_r(G).
\end{CD}
\end{equation}
Since $(\Lambda_{\B}\otimes_\max\lambda)\circ \delta_\B$ is injective, it follows that $\kappa$ is also injective and therefore an isomorphism (so we  just proved (4)$\Rightarrow$(2)).
But then $(\Lambda_B\otimes_\max\id)\circ\delta_\max^r=(\Lambda_{\B}\otimes_\max\lambda)\circ \delta_\B$,
hence $(\Lambda_B\otimes_\max\id)\circ\delta_\max^r$ is injective. Hence (4)$\Rightarrow$(5).
Conversely, if (5) holds, then
diagram (\ref{eq-diag}) implies that $\kappa$ is injective and therefore an isomorphism $C^*(\B)\congto B$. Hence (5)$\Rightarrow$(2).
We saw above that  $(\Lambda_B\otimes_\max\id)\circ\delta_\max^r$, if exists, is faithful. But then
$\delta_\max^r$ and $\delta^r$ are faithful as well.
\end{proof}

\begin{remark}
If $\delta\colon B\to \M(B\otimes C^*(G))$ is a maximal coaction of a locally compact group $G$, then it is Morita equivalent to a dual coaction on a maximal crossed product. Using this it is not difficult to see that such a coaction lifts to $\delta_\max\colon B\to \M(B\otimes_\max C^*(G))$.
But the converse is not true in general and the above theorem does not extend to general
locally compact groups $G$. The problem is that by fundamental results of Choi-Effros and Connes \cite{Choi-Effros, Connes:injective}, the full and reduced group algebras of (almost) connected  second countable groups are always nuclear -- even if the groups are not amenable (like $G=SL_2(\Real)$).
But then it is clear that the dual coaction $\delta_{G,r}:C^*_r(G)\to \M(C^*_r(G)\otimes C^*(G))$,
which is not maximal if $G$ is not amenable,
extends to a faithful map $\delta_{G,r}^\max:C^*_r(G)\to \M(C^*_r(G)\otimes_\max C^*(G))$.
\end{remark}

Notice that the main ingredient in the proof of Theorem~\ref{theo:maximal-discrete} is \cite{Ara-Exel-Katsura:Dynamical_systems}*{Theorem~6.2}, which, as indicated by the authors, is based on an idea of Kirchberg. Part (2) of the following corollary has been already observed in \cite{Ara-Exel-Katsura:Dynamical_systems}. Recall from \cite{Exel:Amenability} that a Fell bundle $\B$ is said to be \emph{amenable} if the regular representation $\Lambda_\B\colon C^*(\B)\to C^*_r(\B)$ is faithful.

\begin{corollary}
Let $G$ be a discrete group.
\begin{enumerate}
\item A Fell bundle $\B$  over $G$ is amenable if and only if the dual coaction
    $\delta_\B$ on $C^*(\B)$ is normal, if and only if the the dual coaction $\delta_{\B,r}$ on $C^*_r(\B)$ is maximal.
\item If the full or reduced cross-sectional \cstar{}algebra of a Fell bundle $\B$ over $G$ is nuclear, then $\B$ is amenable.
\item If $G$ is amenable (that is, if $C^*_r(G)$ or $C^*(G)$ is nuclear), then every Fell bundle over $G$ is amenable.
\end{enumerate}
\end{corollary}
We should point out that the third item has already  been shown  by Exel in  \cite{Exel:Amenability}*{Theorem~4.7}.
\begin{proof}
The first statement follows directly from the fact that the regular representation $\Lambda_\B\colon C^*(\B)\to C^*_r(\B)$
can be thought of as either the normalisation map for $(C^*(\B),\delta_\B)$, or the maximalisation map for $(C^*_r(\B),\delta_{\B,r})$.
The second item follows directly from the combination of (1) and Theorem~\ref{theo:maximal-discrete}. And the third item
also follows from (1) and the fact that every $G$-coaction is maximal and normal for amenable $G$.
\end{proof}

Observe that the converse of (2) above does not hold, that is, there are amenable Fell bundles (over discrete groups) for which $C^*(\B)\cong C^*_r(\B)$ is not nuclear. To see an example let $A$ be any non-nuclear \cstar{}algebra and let an amenable group $G$ act trivially on $A$.
Then $A\rtimes_{\id}G=A\rtimes_{\id,r}G$, hence the corresponding Fell bundle $A\rtimes^\id G$ is amenable.
But $A\rtimes_{\id}G\cong A\otimes C^*(G)$ is not nuclear, since the tensor product of a non-nuclear \cstar{}algebra
with a nuclear \cstar{}algebra is never nuclear. 

If $\B$ is not amenable, we know from the above corollary  that $C^*(\B)$ and $C^*_r(\B)$ are non-nuclear \cstar{}algebras, so there exist \cstar{}algebras $D,E$ such that the algebraic tensor products $C^*(\B)\odot D$ and $C^*(\B)\odot E$ do not admit unique \cstar{}norms. Indeed, the methods above allow us to take explicit choices for $D$ and $E$:

\begin{corollary}
If a Fell bundle $\B$ over a discrete group $G$ is not amenable, then all the algebraic tensor products $C^*(\B)\odot C^*_r(G)$, $C^*_r(\B)\odot C^*(G)$ and $C^*_r(\B)\odot C^*_r(G)$ do not admit unique \cstar{}norms, that is,
\begin{align*}
C^*(\B)\otimes_{\max} C^*_r(G)\not= C^*(\B)\otimes_{\min} C^*_r(G), \\
C^*_r(\B)\otimes_{\max} C^*_r(G)\not= C^*_r(\B)\otimes_{\min} C^*_r(G), \\
C^*_r(\B)\otimes_{\max} C^*(G)\not= C^*_r(\B)\otimes_{\min} C^*(G).
\end{align*}
\end{corollary}
\begin{proof}
The dual coaction $\delta_\B\colon C^*(\B)\to C^*(\B)\otimes C^*(G)$ is maximal and hence its reduction $\delta^r_{\B}\colon C^*(\B)\to C^*(\B)\otimes C^*_r(G)$ lifts to an injective homomorphism $\delta_{\B,\max}^r\colon C^*(\B)\to C^*(\B)\otimes_\max C^*_r(G)$.
But the reduction $\delta_\B^r$ is weakly equivalent to $\Lambda_\B\colon C^*(\B)\to C^*_r(\B)$, hence
if $C^*(\B)\otimes_{\max} C^*_r(G) = C^*(\B)\otimes_{\min} C^*_r(G)$, then $\B$ is amenable. This gives the statement for  $C^*(\B)\odot C^*_r(G)$ and the other cases are treated similarly.
\end{proof}

Taking $\B$ to be the trivial Fell bundle $\B=\C\times G$, the above result gives non-uniqueness of \cstar{}norms on mixed tensor products of the form $C^*(G)\odot C^*_r(G)$ and $C_r^*(G)\otimes C_r^*(G)$. This special case is well-known and follows, for instance, from Proposition~6.4.1 in \cite{Brown-Ozawa:Approximations}.\footnote{{We thank Rufus Willett for pointing out this reference to us.}}
Observe that we do not say anything about the tensor product $C^*(\B)\odot C^*(G)$. Indeed, as
shown in \cite[Proposition 8.1]{Kirchberg:Nonsemisplit}, uniqueness of the \cstar{}norm on $C^*(\B)\odot C^*(G)$
 in the case of $\B=\C\times G$, $G=\Free_\infty$ is equivalent to  Connes' embedding conjecture!



\begin{bibdiv}
 \begin{biblist}

\bib{Abadie:Enveloping}{article}{
  author={Abadie, Fernando},
  title={Enveloping actions and Takai duality for partial actions},
  journal={J. Funct. Anal.},
  volume={197},
  date={2003},
  number={1},
  pages={14--67},
  issn={0022-1236},
  review={\MRref{1957674}{2004c:46130}},
  doi={10.1016/S0022-1236(02)00032-0},
}

\bib{Abadie-Perez:Amenability}{article}{
  author={Abadie, Fernando and Mart{\'{\i }} P{\'e}rez, Laura},
  title={On the amenability of partial and enveloping actions},
  journal={Proc. Amer. Math. Soc.},
  volume={137},
  year={2009},
  number={11},
  pages={3689--3693},
  issn={0002-9939},
  doi={10.1090/S0002-9939-09-09998-5},
  review={\MRref {2529875}{2010f:46099}},
}

\bib{Ara-Exel-Katsura:Dynamical_systems}{article}{
  author={Ara, Pere},
  author={Exel, Ruy},
  author={Katsura, Takeshi},
  title={Dynamical systems of type $(m,n)$ and their $\textup {C}^*$\nobreakdash -algebras},
  journal={Ergodic Theory Dynam. Systems},
  volume={33},
  date={2013},
  number={5},
  pages={1291--1325},
  issn={0143-3857},
  review={\MRref {3103084}{}},
  doi={10.1017/S0143385712000405},
}

\bib{Ara-Exel:K-theory}{article}{
  author={Ara, Pere},
  author={Exel, Ruy},
  title={K-theory for the tame $C^*$-algebra of a separated graph},
  status={preprint, to appear in J. Funct. Analysis},
  date={2015},
  note={\arxiv{1503.06067}},
}

\bib{Baum-Guentner-Willett:Expanders}{article}{
  author={Baum, Paul},
  author={Guentner, Erik},
  author={Willett, Rufus},
  title={Expanders, exact crossed products, and the Baum-Connes conjecture},
  status={eprint},
  note={\arxiv {1311.2343}},
    journal={Annals of K-Theory (to appear)},
}

\bib{BG}{article}{
    AUTHOR = {Brown, Nathanial P.},
    author= {Guentner, Erik P.},
     TITLE = {New {$\rm C^\ast$}-completions of discrete groups and
              related spaces},
   JOURNAL = {Bull. Lond. Math. Soc.},
  FJOURNAL = {Bulletin of the London Mathematical Society},
    VOLUME = {45},
      YEAR = {2013},
    NUMBER = {6},
     PAGES = {1181--1193},
      ISSN = {0024-6093},
   MRCLASS = {46L05 (22D20 22D25)},
  MRNUMBER = {3138486},
MRREVIEWER = {Takahiro Sudo},
       DOI = {10.1112/blms/bdt044},
       URL = {http://dx.doi.org/10.1112/blms/bdt044},
}

\bib{Brown-Ozawa:Approximations}{book}{
  author={Brown, Nathanial P.},
  author={Ozawa, Narutaka},
  title={$C^*$\nobreakdash-algebras and finite-dimensional approximations},
  series={Graduate Studies in Mathematics},
  volume={88},
  publisher={American Mathematical Society, Providence, RI},
  date={2008},
  pages={xvi+509},
  isbn={978-0-8218-4381-9},
  isbn={0-8218-4381-8},
  review={\MRref{2391387}{2009h:46101}},
}

\bib{Buss:thesis}{thesis}{
  author={Buss, Alcides},
  title={Generalized fixed point algebras for coactions of locally compact quantum groups},
  date={2007},
  institution={Univ. Münster},
  type={Doctoral Thesis},
}

\bib{Buss-Echterhoff:Exotic_GFPA}{article}{
  author={Buss, Alcides},
  author={Echterhoff, Siegfried},
  title={Universal and exotic generalized fixed-point algebras for weakly proper actions and duality},
  journal={Indiana Univ. Math. Journal},
  volume={63},
  year={2014},
  number={6},
  pages={1659--1701},
  issn={0022-2518},
  note={\arxiv {1304.5697}},
}

\bib{Buss-Echterhoff-Willett:Exotic}{article}{
  author={Buss, Alcides},
  author={Echterhoff, Siegfried},
  author={Rufus Willett},
  title={Exotic crossed products and the Baum-Connes conjecture},
  status={eprint},
  note={\arxiv {1409.4332}},
  journal ={Journal reine und angew. Math. (to appear)},
}

\bib{Choi-Effros}{article}{
   AUTHOR = {Choi, Man Duen},
   author=  {Effros, Edward G.},
     TITLE = {Separable nuclear {$C\sp*$}-algebras and injectivity},
   JOURNAL = {Duke Math. J.},
  FJOURNAL = {Duke Mathematical Journal},
    VOLUME = {43},
      YEAR = {1976},
    NUMBER = {2},
     PAGES = {309--322},
      ISSN = {0012-7094},
   MRCLASS = {46L05},
  MRNUMBER = {0405117 (53 \#8912)},
MRREVIEWER = {Christopher Lance},
}

\bib{Connes:injective}{article}{
    AUTHOR = {Connes, Alain},
     TITLE = {Classification of injective factors. {C}ases {$II_{1},$}
              {$II_{\infty },$} {$III_{\lambda },$} {$\lambda \not=1$}},
   JOURNAL = {Ann. of Math. (2)},
  FJOURNAL = {Annals of Mathematics. Second Series},
    VOLUME = {104},
      YEAR = {1976},
    NUMBER = {1},
     PAGES = {73--115},
      ISSN = {0003-486X},
   MRCLASS = {46L10},
  MRNUMBER = {0454659 (56 \#12908)},
MRREVIEWER = {Francois Combes},
}

\bib{Cuntz:K-amenable}{article}{
   AUTHOR = {Cuntz, Joachim},
     TITLE = {{$K$}-theoretic amenability for discrete groups},
   JOURNAL = {J. Reine Angew. Math.},
  FJOURNAL = {Journal f\"ur die Reine und Angewandte Mathematik},
    VOLUME = {344},
      YEAR = {1983},
     PAGES = {180--195},
      ISSN = {0075-4102},
     CODEN = {JRMAA8},
   MRCLASS = {46L80 (19K99)},
  MRNUMBER = {716254 (86e:46064)},
MRREVIEWER = {Autorreferat},
       DOI = {10.1515/crll.1983.344.180},
       URL = {http://dx.doi.org/10.1515/crll.1983.344.180},
}

\bib{Doran-Fell:Representations}{book}{
  author={Doran, Robert S.},
  author={Fell, James M. G.},
  title={Representations of $^*$\nobreakdash -algebras, locally compact groups, and Banach $^*$\nobreakdash -algebraic bundles. Vol. 1},
  series={Pure and Applied Mathematics},
  volume={125},
  publisher={Academic Press Inc.},
  place={Boston, MA},
  date={1988},
  pages={xviii+746},
  isbn={0-12-252721-6},
  review={\MRref {936628}{90c:46001}},
}

\bib{Doran-Fell:Representations_2}{book}{
  author={Doran, Robert S.},
  author={Fell, James M. G.},
  title={Representations of $^*$\nobreakdash -algebras, locally compact groups, and Banach $^*$\nobreakdash -algebraic bundles. Vol. 2},
  series={Pure and Applied Mathematics},
  volume={126},
  publisher={Academic Press Inc.},
  place={Boston, MA},
  date={1988},
  pages={i--viii and 747--1486},
  isbn={0-12-252722-4},
  review={\MRref {936629}{90c:46002}},
}

\bib{Echterhoff:Morita_twisted}{article}{
  author={Echterhoff, Siegfried},
  title={Morita equivalent twisted actions and a new version of the Packer--Raeburn stabilization trick},
  journal={J. London Math. Soc. (2)},
  volume={50},
  date={1994},
  number={1},
  pages={170--186},
  issn={0024-6107},
  review={\MRref{1277761}{96a:46118}},
  doi={10.1112/jlms/50.1.170},
}

\bib{Echterhoff-Quigg:InducedCoactions}{article}{
  author={Echterhoff, Siegfried},
  author={Quigg, John},
  title={Induced coactions of discrete groups on $C^*$\nobreakdash -algebras},
  journal={Canad. J. Math.},
  volume={51},
  date={1999},
  number={4},
  pages={745--770},
  issn={0008-414X},
  review={\MRref {1701340}{2000k:46094}},
  doi={10.4153/CJM-1999-032-1},
}

\bib{Echterhoff-Kaliszewski-Quigg:Maximal_Coactions}{article}{
  author={Echterhoff, Siegfried},
  author={Kaliszewski, Steven P.},
  author={Quigg, John},
  title={Maximal coactions},
  journal={Internat. J. Math.},
  volume={15},
  date={2004},
  number={1},
  pages={47--61},
  issn={0129-167X},
  doi={10.1142/S0129167X04002107},
  review={\MRref {2039211}{2004j:46087}},
}

\bib{Echterhoff-Kaliszewski-Quigg-Raeburn:Categorical}{article}{
  author={Echterhoff, Siegfried},
  author={Kaliszewski, Steven P.},
  author={Quigg, John},
  author={Raeburn, Iain},
  title={A categorical approach to imprimitivity theorems for $C^*$\nobreakdash -dynamical systems},
  journal={Mem. Amer. Math. Soc.},
  volume={180},
  date={2006},
  number={850},
  pages={viii+169},
  issn={0065-9266},
  review={\MRref {2203930}{2007m:46107}},
  doi={10.1090/memo/0850},
}

\bib{Exel:Circle}{article}{
  author={Exel, Ruy},
  title={Circle actions on $C^*$\nobreakdash-algebras, partial automorphisms, and a generalized Pimsner--Voiculescu exact sequence},
  journal={J. Funct. Anal.},
  volume={122},
  date={1994},
  number={2},
  pages={361--401},
  issn={0022-1236},
  review={\MRref{1276163}{95g:46122}},
  doi={10.1006/jfan.1994.1073},
}

\bib{Exel:TwistedPartialActions}{article}{
  author={Exel, Ruy},
  title={Twisted partial actions: a classification of regular $C^*$\nobreakdash -algebraic bundles},
  journal={Proc. London Math. Soc. (3)},
  volume={74},
  date={1997},
  number={2},
  pages={417--443},
  issn={0024-6115},
  review={\MRref {1425329}{98d:46075}},
  doi={10.1112/S0024611597000154},
}

\bib{Exel:Amenability}{article}{
  author={Exel, Ruy},
  title={Amenability for Fell bundles},
  journal={J. Reine Angew. Math.},
  volume={492},
  date={1997},
  pages={41--73},
  issn={0075-4102},
  review={\MRref {1488064}{99a:46131}},
  doi={10.1515/crll.1997.492.41},
}

\bib{ExelNg:ApproximationProperty}{article}{
  author={Exel, Ruy},
  author={Ng, {Ch}i-Keung},
  title={Approximation property of $C^*$\nobreakdash -algebraic bundles},
  journal={Math. Proc. Cambridge Philos. Soc.},
  volume={132},
  date={2002},
  number={3},
  pages={509--522},
  issn={0305-0041},
  doi={10.1017/S0305004101005837},
  review={\MRref {1891686}{2002k:46189}},
}

\bib{Exel:Book}{book}{
  author={Exel, Ruy},
  title={Partial Dynamical Systems, Fell Bundles and Applications},
  note={To appear in NYJM book series},
  status={preprint},
  eprint={http://mtm.ufsc.br/~exel/papers/pdynsysfellbun.pdf},
}

\bib{Fell:Induced}{book}{
  author={Fell, James M. G.},
  title={Induced representations and {B}anach {$\sp *$}-algebraic bundles},
  series={Lecture Notes in Mathematics, Vol. 582},
  note={With an appendix due to A. Douady and L. Dal Soglio-H{\'e}rault},
  publisher={Springer-Verlag, Berlin-New York},
  year={1977},
  pages={iii+349},
  review={\MRref {0457620}{56 \#15825}},
}

\bib{IT}{article}{
  AUTHOR = {Imai, Sh{\=o}},
  author =  {Takai, Hiroshi},
     TITLE = {On a duality for {$C^{\ast} $}-crossed products by a locally
              compact group},
   JOURNAL = {J. Math. Soc. Japan},
  FJOURNAL = {Journal of the Mathematical Society of Japan},
    VOLUME = {30},
      YEAR = {1978},
    NUMBER = {3},
     PAGES = {495--504},
     CODEN = {NISUBC},
   MRCLASS = {46L55 (22D35)},
  MRNUMBER = {500719 (81h:46090)},
       DOI = {10.2969/jmsj/03030495},
       URL = {http://dx.doi.org/10.2969/jmsj/03030495},
}
	
\bib{JV}{article}{
    AUTHOR = {Julg, Pierre},
    author=  {Valette, Alain},
     TITLE = {{$K$}-theoretic amenability for {${\rm SL}_{2}({\bf
              Q}_{p})$}, and the action on the associated tree},
   JOURNAL = {J. Funct. Anal.},
  FJOURNAL = {Journal of Functional Analysis},
    VOLUME = {58},
      YEAR = {1984},
    NUMBER = {2},
     PAGES = {194--215},
      ISSN = {0022-1236},
     CODEN = {JFUAAW},
   MRCLASS = {22E50 (18F25 19K99 46L80 46M20 58G12)},
  MRNUMBER = {757995 (86b:22030)},
MRREVIEWER = {Alan L. T. Paterson},
       DOI = {10.1016/0022-1236(84)90039-9},
       URL = {http://dx.doi.org/10.1016/0022-1236(84)90039-9},
}

\bib{Kaliszewski-Landstad-Quigg:Exotic}{article}{
  author={Kaliszewski, Steven P.},
  author={Landstad, Magnus B.},
  author={Quigg, John},
  title={Exotic group $C^*$\nobreakdash -algebras in noncommutative duality},
  journal={New York J. Math.},
  volume={19},
  date={2013},
  pages={689--711},
  issn={1076-9803},
  review={\MRref {3141810}{}},
  eprint={http://nyjm.albany.edu/j/2013/19_689.html},
}

\bib{Kaliszewski-Landstad-Quigg:Exotic-coactions}{article}{
  author={Kaliszewski, Steven P.},
  author={Landstad, Magnus B.},
  author={Quigg, John},
  title={Exotic coactions},
  status={eprint},
  date={2013},
  note={\arxiv {1305.5489}},
}

\bib{Kaliszewski-Muhly-Quigg-Williams:Coactions_Fell}{article}{
  author={Kaliszewski, Steven P.},
  author={Muhly, Paul S.},
  author={Quigg, John},
  author={Williams, Dana P.},
  title={Coactions and Fell bundles},
  journal={New York J. Math.},
  volume={16},
  date={2010},
  pages={315--359},
  issn={1076-9803},
  review={\MRref {2740580}{2012d:46165}},
  eprint={http://nyjm.albany.edu/j/2010/16-13v.pdf},
}

\bib{Kirchberg:Nonsemisplit}{article}{
  author={Kirchberg, Eberhard},
  title={On nonsemisplit extensions, tensor products and exactness of group $C^*$\nobreakdash -algebras},
  journal={Invent. Math.},
  volume={112},
  date={1993},
  number={3},
  pages={449--489},
  issn={0020-9910},
  doi={10.1007/BF01232444},
  review={\MRref {1218321}{94d:46058}},
}

\bib{McClanahan:K-theory-partial}{article}{
  author={McClanahan, Kevin},
  title={$K$\nobreakdash-theory for partial crossed products by discrete groups},
  journal={J. Funct. Anal.},
  volume={130},
  date={1995},
  number={1},
  pages={77--117},
  issn={0022-1236},
  review={\MRref{1331978}{96i:46083}},
  doi={10.1006/jfan.1995.1064},
}

\bib{Nilsen}{article}{
   AUTHOR = {Nilsen, May},
     TITLE = {Duality for full crossed products of {$C^\ast$}-algebras by
              non-amenable groups},
   JOURNAL = {Proc. Amer. Math. Soc.},
  FJOURNAL = {Proceedings of the American Mathematical Society},
    VOLUME = {126},
      YEAR = {1998},
    NUMBER = {10},
     PAGES = {2969--2978},
      ISSN = {0002-9939},
     CODEN = {PAMYAR},
   MRCLASS = {46L55},
  MRNUMBER = {1469427 (99a:46120)},
MRREVIEWER = {John Quigg},
       DOI = {10.1090/S0002-9939-98-04598-5},
       URL = {http://dx.doi.org/10.1090/S0002-9939-98-04598-5},
}

\bib{Ng:Discrete-coactions}{article}{
  author={Ng, Chi-Keung},
  title={Discrete coactions on {$C\sp \ast $}-algebras},
  journal={J. Austral. Math. Soc. Ser. A},
  volume={60},
  year={1996},
  number={1},
  pages={118--127},
  review={\MRref {1364557}{97a:46093}},
}

\bib{Okayasu:Free-group}{article}{
  author={Okayasu, Rui},
  title={Free group \cstar{}algebras associated with $\ell _p$},
  journal={Internat. J. Math.},
  volume={25},
  year={2014},
  number={7},
  pages={1450065 (12 pages)},
  issn={0129-167X},
  doi={10.1142/S0129167X14500657},
  review={\MRref {3238088}{}},
}

\bib{Packer-Raeburn:Stabilisation}{article}{
  author={Packer, Judith A.},
  author={Raeburn, Iain},
  title={Twisted crossed products of $C^*$\nobreakdash-algebras},
  journal={Math. Proc. Cambridge Philos. Soc.},
  volume={106},
  date={1989},
  pages={293--311},
  review={\MRref{1002543}{90g:46097}},
  doi={10.1017/S0305004100078129},
}

\bib{Quigg:Discrete-coactions}{article}{
  author={Quigg, John C.},
  title={Discrete \cstar{}coactions and \cstar{}algebraic bundles},
  journal={J. Austral. Math. Soc. Ser. A},
  volume={60},
  year={1996},
  number={2},
  pages={204--221},
  issn={0263-6115},
  review={\MRref {1375586}{97c:46086}},
}

\bib{Quigg:Landstad_duality}{article}{
  author={Quigg, John C.},
  title={Landstad duality for $C^*$\nobreakdash -coactions},
  journal={Math. Scand.},
  volume={71},
  date={1992},
  number={2},
  pages={277--294},
  issn={0025-5521},
  review={\MRref {1212711}{94e:46119}},
  eprint={http://www.mscand.dk/article.php?id=956},
}

\bib{Rieffel:Proper}{article}{
  author={Rieffel, Marc A.},
  title={Proper actions of groups on $C^*$\nobreakdash -algebras},
  conference={ title={Mappings of operator algebras}, address={Philadelphia, PA}, date={1988}, },
  book={ series={Progr. Math.}, volume={84}, publisher={Birkh\"auser Boston}, place={Boston, MA}, },
  date={1990},
  pages={141--182},
  review={\MRref {1103376}{92i:46079}},
}

\bib{Rieffel:Integrable_proper}{article}{
  author={Rieffel, Marc A.},
  title={Integrable and proper actions on $C^*$\nobreakdash -algebras, and square-integrable representations of groups},
  journal={Expo. Math.},
  volume={22},
  date={2004},
  number={1},
  pages={1--53},
  issn={0723-0869},
  review={\MRref {2166968}{2006g:46108}},
  doi={10.1016/S0723-0869(04)80002-1},
}

\bib{Sehnem:MasterThesis}{thesis}{
  author={Sehnem, Camila F.},
  title={Uma classificação de fibrados de Fell estáveis},
  date={2014},
  institution={Universidade Federal de Santa Catarina},
  type={Master's Thesis},
  eprint={http://mtm.ufsc.br/pos/Camila_Fabre_Sehnem.pdf},
}

\bib{Tu}{article}{
  AUTHOR = {Tu, Jean-Louis},
     TITLE = {The {B}aum-{C}onnes conjecture and discrete group actions on
              trees},
   JOURNAL = {$K$-Theory},
  FJOURNAL = {$K$-Theory. An Interdisciplinary Journal for the Development,
              Application, and Influence of $K$-Theory in the Mathematical
              Sciences},
    VOLUME = {17},
      YEAR = {1999},
    NUMBER = {4},
     PAGES = {303--318},
      ISSN = {0920-3036},
     CODEN = {KTHEEO},
   MRCLASS = {19K35 (46L80)},
  MRNUMBER = {1706113 (2000h:19003)},
MRREVIEWER = {Yuri A. Kordyukov},
       DOI = {10.1023/A:1007751625568},
       URL = {http://dx.doi.org/10.1023/A:1007751625568},
}

 \end{biblist}
\end{bibdiv}
\end{document}